\documentclass{article}
\usepackage{latexsym, amssymb}
\usepackage{graphicx}
\usepackage{amsmath, amsbsy}
\usepackage{amsopn, amstext}
\usepackage[colorlinks,linkcolor=blue,hyperindex]{hyperref}
\usepackage[textwidth=16cm,textheight=21.5cm]{geometry}

\usepackage{tikz}
\usetikzlibrary{arrows}
\DeclareMathOperator{\rank}{rank}\DeclareMathOperator{\im}{im}
 \DeclareMathOperator{\diag}{diag}

\label{key}

\def\D{\Delta}

\def\s{\sigma}

\newcommand{\R}{{\mathbb R}}

\newcommand{\dt}{\delta}
\newcommand{\Dt}{\Delta}

\newcommand{\Na}{\mathcal{N}}
\newcommand{\Gr}{\mathcal{G}}

\newcommand{\Po}{\mathcal{P}}
\newcommand{\Ha}{\mathcal{H}}

\newtheorem{theorem}{Theorem}[section]

\newtheorem{corollary}[theorem]{Corollary}
\newtheorem{definition}{Definition}
\newtheorem{proposition}[theorem]{Proposition}

\newtheorem{example}[theorem]{Example}

\newtheorem{remark}{Remark}[section]
\newtheorem{algo}{Algorithm}
\usepackage{times,amsmath,amssymb,graphicx,tikz}
\usetikzlibrary{automata,arrows,positioning}

\usepackage{algorithm,algorithmic}

\begin{document}

\begin{center}

{\Large\bf Polynomial representation for orthogonal projections onto subspaces of finite games
\footnote{
}}
\end{center}

\vskip 2mm

\begin{center}
	{\Large Kuize Zhang \dag\ddag }

\vskip 2mm

\textsuperscript{\dag} College of Automation, Harbin Engineering University, Harbin 150001, P.~R.~China\\
Email: zkz0017@163.com\\
\textsuperscript{\ddag}Institute of  Systems Science,
Chinese Academy of Sciences, Beijing 100190, P.~R.~China\\

\vskip 2mm

\end{center}

\vskip \baselineskip

\underline{\bf Abstract.}
The space of finite games can be decomposed into three orthogonal subspaces
\cite{Candogan2011FlowsDecompositionofGames}, which are
the subspaces of pure potential games, nonstrategic games and pure harmonic games.
The orthogonal projections onto these subspaces are represented as the Moore-Penrose inverses
of the corresponding linear operators (i.e., matrices) \cite{Candogan2011FlowsDecompositionofGames}.
Although the representation is compact and nice, no analytic method is given to
calculate Moore-Penrose inverses of these linear operators.
Hence using their results, one cannot verify whether a finite game belongs to
one of these subspaces.
In this paper, jumping over calculating Moore-Penrose inverses of these linear operators directly,
via using group inverses,
in the framework of the semitensor product of matrices,
we give explicit polynomial representation for these orthogonal projections and
for potential functions of potential games.
Using our results,  one not only can determine whether a finite game belongs to
one of these subspaces, but also can find an arbitrary finite game belonging to one of them.
Besides, we give formal definitions for these types of games by using their payoff functions.
Based on these results, more properties of finite games are revealed.

\vskip 2mm



\vskip \baselineskip

\section{Introduction}
\label{sec:Introduction}

Rosenthal \cite{Rosenthal1973PotentialGames} initiates the concept of potential games, and proves that
every potential game has a pure Nash equilibrium.
Monderer and Shapley \cite{Monderer1996PotentialGames} systematically investigate potential games,
give a method to verify whether a given game is potential, and prove that every potential game
is isomorphic to a congestion game.  Intuitively speaking, a potential game is a game with a function
from the set of strategy profiles to the reals satisfying that, the deviation of every player's payoff caused
by the deviation of this player's strategies is equal to that of the function also caused by
the deviation of this player's strategies. Partially due to the fact that
there is one common function describing the deviation of
every player's payoff, potential games have been applied to many fields, e.g.,
traffic networks
\cite{Marden2009JSFPwithInertiaPotentialGames,Xiao2013ASFP,Wang2013DistributedConsensusCongestionGame},
cooperative control \cite{Marden2009CooperativeControlPotentialGame},
optimization of distributed coverage of graphs \cite{Zhu2013DistributedCoverageGame},
etc..

Although potential games possess so good properties and wide applications, there
are other types of games that are not potential but still have good properties and potential
applications. For example, the Rock-Paper-Scissors game is not potential,
but has the uniformly mixed strategy profile
as a mixed Nash equilibrium \cite{Candogan2011FlowsDecompositionofGames}. Hence
it is necessary to give a systematic characterization for these finite games to
investigate properties of other types of games and find their practical applications.
When the number of players and the numbers of their strategies are fixed,
Candogan et al. \cite{Candogan2011FlowsDecompositionofGames} identify the set of finite games with a finite-dimensional
Euclidean space, and decompose this space into three orthogonal subspaces as follows:
\begin{equation}
\rlap{$\underbrace{\phantom{\quad{\mathcal P}\quad\oplus\quad{\mathcal N}}}_{Potential\quad games}$}\quad{\mathcal P}\quad\oplus\quad
\overbrace{{\mathcal N}\quad\oplus\quad{\mathcal H}}^{Harmonic\quad games},
\end{equation}
where these subspaces are (\romannumeral1) the pure potential subspace ${\mathcal P}$,
(\romannumeral2) the nonstrategic subspace ${\mathcal N}$,
and  (\romannumeral3) the pure harmonic subspace ${\mathcal H}$.
It is also demonstrated in \cite{Candogan2011FlowsDecompositionofGames}
that the pure potential subspace plus the nonstrategic subspace
is the potential subspace, denoted as
${\mathcal G}_P={{\mathcal P}}\oplus {{\mathcal N}}$;
and the pure harmonic subspace plus the nonstrategic subspace is the harmonic subspace, denoted as $
{\mathcal G}_H={{\mathcal H}}\oplus {{\mathcal N}}$. Nonstrategic games are such that
every strategy profile is a pure Nash equilibirium.  Harmonic games generically do not have
pure Nash equilibria, but always have the uniformly mixed strategy profiles as mixed Nash
equilibria. Actually Candogan et al. \cite{Candogan2011FlowsDecompositionofGames}
show orthogonal projections onto these subspaces
by using
the {\it Moore-Penrose inverses} of the corresponding linear operators (i.e., matrices).
Although these orthogonal projections are very compact and nice, one cannot use them to
verify whether a finite game belongs to one of these subspaces, because it is not shown how to
calculate the Moore-Penrose inverses of these linear operators in
\cite{Candogan2011FlowsDecompositionofGames}.
It is well known that Moore-Penrose inverses
of matrices can be represented via singular value decompositions of matrices (details are seen
in Subsection \ref{subsec:GI}), and it is impossible to calculate singular value decompositions
of matrices analytically.
Although the corresponding matrices related to decomposition of finite games have special structure,
to the best of our knowledge, there has been no analytic method to calculate their singular value decompositions.
Hence it is very difficult to obtain explicit forms of these orthogonal projections.
It is important to obtain explicit forms of these projections, because if they were obtained,
one could verify whether a given finite game belongs to one of these subspaces, and furthermore,
one could find an arbitrary finite game belonging to one of these subspaces.
Hence explicit forms are very important for further study on properties of games belong to
these subspaces.
In this paper,  we not only show how to calculate these orthogonal projections,
but also give closed (actually polynomial) forms  for them.
Using these results, one not only can verify whether a finite game belongs to one of
these subspaces, but also can construct an arbitrary game belonging to one of these subspaces
(Theorem \ref{thm17:game_decomposition}).
The advantage of Theorem \ref{thm17:game_decomposition} is that
it shows formulae for these orthogonal projections as functions of the number of players
and those of their strategies. Hence no matter whether the
number of players and those of their strategies are given,
one knows what these orthogonal projections ``look like''.

Our results are given in the framework of the {\it semitensor product (STP) of matrices} built
by Cheng \cite{Cheng2014PotentialGame}, in which a linear equation (called potential equation) is
proposed such that a finite game is potential iff the potential equation has a solution; and it is also proved that
if the potential equation has a solution, then the potential function
of the corresponding game can be calculated from any solution. In this paper, based on Theorem
\ref{thm17:game_decomposition}, we will show a further result
(Theorem \ref{thm4:game_decomposition}), i.e., closed (actually polynomial) forms for both the solutions of
potential equations and the corresponding potential functions. Using our results,
to testify whether a finite game is potential
and to calculate its potential function, one only needs to
calculate the multiplication and addition of matrices only depending on the number of players
and those of their strategies.
Specifically, when the number of players and those of their strategies are given, we
can show the orthogonal projection $P_{\Gr_P}$ onto the subspace of potential games,
then every game $u$ with the same number of players and those of their strategies
is potential iff $P_{\Gr_P}u=u$ (details are seen in Section \ref{sec:explicitProjontoGameSubsapce}).
Note that $P_{\Gr_P}$ does not depend on $u$,
but the potential equation depends on $u$ \cite{Cheng2014PotentialGame}.
So our result is an essential improvement compared to \cite{Cheng2014PotentialGame}.
The STP of matrices is for the first time proposed by Cheng \cite{Cheng2001STP}.
STP is a natural generalization of the conventional matrix product,
and has been applied into many fields, e.g.,
analysis and control of Boolean control networks
\cite{Cheng(2011book),Laschov2013ObservabilityofBN:GraphApproach,Fornasini2013ObservabilityReconstructibilityofBCN,Zou2015KalmanDecompositionBCN,Zhang2015Invertibility_BCN}, 
control-theoretic problems
\cite{Cheng2001STP}, symmetry of dynamical systems \cite{Cheng2007NonlinearSys.LinearSymmetry},
differential geometry and Lie algebras \cite{Cheng2003STPtoDifferentialGeometryLieAlgebra}, etc..
Basic knowledge on STP is shown in Section \ref{sec:ERofFGS:preliminary}.

The key tool that we use to obtain the polynomial representation
of these orthogonal projections and potential functions is {\it group inverse}.
Moore-Penrose inverses and group inverses are both generalized inverses.
Actually, the main results of this paper are obtained from a result given in our previous
paper \cite{Zhang(2012)}. Basic knowledge is seen in Subsection \ref{subsec:GI}.

The remainder parts of this paper are arranged as follows.
Section \ref{sec:preli} introduces necessary basic knowledge on finite-dimensional Euclidean spaces,
Moore-Penrose inverses and group inverses.
Section \ref{sec:ERofFGS:preliminary} introduces noncooperative strategic form  finite games
and their vector space structure in the framework of STP.
 Section \ref{sec:ProjontoGameSubsapce} shows orthogonal projections
onto subspaces of finite games based on Moore-Penrose inverses in the framework of STP.
Section \ref{sec:explicitProjontoGameSubsapce} shows the main contribution of this paper, polynomial
representation for these orthogonal projections and for potential functions of potential games.
Some examples are also shown to demonstrate the advantage of our results.
Section \ref{sec:conclusion} ends up with some remar ks.

%

\section{Preliminaries}\label{sec:preli}

In this section, necessary basic knowledge on Euclidean spaces, Moore-Penrose inverses and group inverses of matrices
are introduced. Notations are first shown as below.

\subsection{Notations}

\begin{itemize}
  \item $2^A$: the power set of set $A$
  \item $\im(A)$ (resp. $\ker(A)$): the image (resp. kernel) space of
	  a matrix $A$
  \item $\R$: the set of real numbers (the reals for short)
  \item $\R^m$: the set of $m$-dimensional real column vector space
  \item $\R_{m\times n}$:  the set of $m\times n$ real matrices
  \item $I_n$: the $n\times n$ identity matrix
  \item $\delta_n^i$: the $i$-th column of the identity matrix $I_n$
  \item $\Delta_n$: the set of columns of $I_n$
  \item $[1,p]$: the first $p$ positive integers
  \item $A^T$: the transpose of matrix $A$
  \item ${\bf 1}_k$: $(\underbrace{1,\dots,1}_{k})^T$
  \item ${\bf 1}_{m\times n}$ ($0_{m\times n}$): the $m\times n$ matrix
	  with all entries equal to $1$ ($0$)
  \item $|A|$: the cardinality of set $A$
  \item $A^{\dag}$: the Moore-Penrose inverse of matrix $A$
  \item $A^{\sharp}$: the group inverse of square matrix $A$
  \item $C_n^k$: the binomial coefficient $\frac{n!}{k!(n-k)!}$
  \item $A_1\oplus A_2\oplus\cdots\oplus A_n$: $\begin{bmatrix}
		  A_1 & 0 & \cdots & 0\\
		  0   & A_2 & \cdots & 0\\
		  \vdots & \vdots & \ddots & \vdots\\
		  0   & 0   & \cdots  & A_n
		  \end{bmatrix}$
\end{itemize}

\subsection{Euclidean spaces and orthogonality}

In this paper, we consider the Euclidean space $\R^m$ with the conventional
inner product: for all $x,y\in\R^{m}$,
$\langle x,y\rangle=x^Ty$. Next we introduce some necessary preliminaries
on orthogonality and projections. These results can be found in many textbooks
on matrix theory, e.g., \cite{Horn2012MatrixAnalysis}.
Two vectors $x,y\in\R^m$ are called orthogonal,
denoted by $x\bot y$, if $\langle x,y\rangle=0$.
Two subspaces $\mathcal{A,B}\subset \R^m$ are called orthogonal, also denoted
by $\mathcal{A}\bot\mathcal{B}$, if for all $x\in\mathcal{A}$, all $y\in\mathcal{B}$,
$\langle x,y\rangle=0$. The unique orthogonal component subspace of $\mathcal{A}$
is denoted by $\mathcal{A}^{\bot}$.
An idempotent symmetric matrix $A\in\R_{m\times m}$
is called the orthogonal projection of $\R^m$ onto $\im(A)$.
Note that if $A\in\R_{m\times m}$ is a projection
(i.e., idempotent), then $x\in\im(A)$ iff $x=Ax$.
The following helpful proposition will be used in the main results.

\begin{proposition}\label{prop_orthogonalprojection1}
	Given two orthogonal projections
	$A_1,A_2\in\R_{m\times m}$. $\im(A_1)=\im(A_2)$ iff
	$A_1=A_2$.
\end{proposition}

\begin{proof}
	The ``if'' part holds obviously.
	Next we prove the ``only if'' part.
	For all $x\in\R^m$, $A_1x\in\im(A_1)=\im(A_2)$, then
	$A_2A_1x=A_1x$. Hence $A_2A_1=A_1$. Symmetrically $A_1A_2=A_2$.
	Then $A_1=A_1^T=(A_2A_1)^T=A_1^TA_2^T=A_1A_2=A_2$.
\end{proof}

\subsection{Moore-Penrose inverses and group inverses}\label{subsec:GI}

In this subsection we introduce necessary basic knowledge on Moore-Penrose inverses
and group inverses (two types of generalized inverses of matrices).
The following Propositions \ref{prop_MoorePenroseinverse},
\ref{prop_linearspace} and \ref{prop_groupinverse} over the complex field
can be found in \cite{Wang2004GeneralizedInverseBook}. Their current version
over the real field can be proved similarly
by using the singular value
decomposition of matrices over the real field.
Their proofs are omitted.

For a matrix $A\in\R_{m\times n}$, a matrix $X\in\R_{n\times m}$ satisfying
$AXA=A,XAX=X,(AX)^T=AX,(XA)^T=XA$ is called the Moore-Penrose inverse
of $A$, and is denoted by $X=A^{\dag}$.

\begin{proposition}\label{prop_MoorePenroseinverse}
	A matrix $A\in\R_{m\times n}$ has one and only one Moore-Penrose inverse.
\end{proposition}

Note that for any matrix $A\in\R_{m\times n}$ with its singular value decomposition
$$A=U\begin{bmatrix}
	\diag(\lambda_1,\dots,\lambda_r) & 0\\
	0 & 0
\end{bmatrix}V^T,$$
where each $\lambda_i$ is a positive singular value of $A$, $U\in\R_{m\times m}$
and $V\in\R_{n\times n}$ are both orthogonal matrices,
$$A^{\dag}=V\begin{bmatrix}
	\diag(\lambda_1^{-1},\dots,\lambda_r^{-1}) & 0\\
	0 & 0
\end{bmatrix}U^T.$$
It is well known that it is impossible to calculate $\lambda_i$'s analytically,
hence it is impossible to calculate $A^{\dag}$ analytically based on the singular value decomposition.

\begin{proposition}\label{prop_linearspace}
	For every matrix $A\in\R_{m\times n}$,
	$\R^m=\im(A)\oplus\ker(A^T)=\im(AA^{\dag})\oplus\im(I_m-AA^{\dag})$,
	$\im(A)=\im(AA^T)=\im(AA^{\dag})$, $\ker(A^T)=\ker(AA^T)=
	\im(I_m-AA^{\dag})$,
	$AA^{\dag}$ (resp. $I_m-AA^{\dag}$) is the orthogonal projection
	of $\R^m$ onto
	$\im(A)$ (resp. $\ker(A^{T})$). 
\end{proposition}

For a matrix $A\in\R_{n\times n}$, a matrix $X\in\R_{n\times n}$ satisfying
$AXA=A,XAX=X,AX=XA$ is called the group inverse
of $A$, and is denoted by $X=A^{\sharp}$.

\begin{proposition}\label{prop_groupinverse}
	A matrix $A\in\R_{n\times n}$ has at most one group inverse, and
	$A$ has a group inverse iff $\rank(A)=\rank(A^2)$.
	If $A$ has a group inverse $A^{\sharp}$, then $A^{\sharp}$ is
	a polynomial of $A$.
\end{proposition}

By Proposition \ref{prop_groupinverse},
every symmetric square real matrix has a group inverse.
By Proposition \ref{prop_linearspace},
for every matrix $A\in\R_{m\times n}$,
$AA^T(AA^T)^{\sharp}$ is the orthogonal projection
of $\R^m$ onto $\im(A)$, 
because $(AA^T)^{\sharp}=(AA^{T})^{\dag}$.
Then the following Proposition \ref{prop_groupMPinverse} holds.

\begin{proposition}\label{prop_groupMPinverse}
	For every matrix $A\in\R_{m\times n}$, $AA^{\dag}=AA^T(AA^T)^{\sharp}=
	A(A^TA)^{\sharp}A^T$.
\end{proposition}

%

Proposition \ref{prop_groupinverse_Oredomain} is a special case of
our previous result \cite[Theorem 4.1]{Zhang(2012)}, and is
a central proposition that leads to the main results of the current paper.

\begin{proposition}[\cite{Zhang(2012)}]\label{prop_groupinverse_Oredomain}
	A matrix $A\in\R_{n\times n}$ has
	a group inverse iff there is a matrix $X\in\R_{n\times n}$ such that
	$A^2X=A$. If $A$ has a group inverse, then for every $X$ satisfying
	$A^2X=A$, $A^{\sharp}=AX^2$.
\end{proposition}

The following Proposition
\ref{prop_MPinverse_rangeinclusion} will also play an important role
in the main results.

\begin{proposition}\label{prop_MPinverse_rangeinclusion}
	Let $A\in\R_{p\times q}$ and $B\in\R_{p\times r}$ and $\im(A)\supset
	\im(B)$. Then $BB^{\dag}AA^{\dag}=AA^{\dag}BB^{\dag}=BB^{\dag}$.
\end{proposition}

\begin{proof}
	Note that $BB^{\dag}$ (resp. $AA^{\dag}$)
	is an orthogonal projection of $\R^p$ onto $\im(B)$
	(resp. $\im(A)$).

	For an arbitrary $x\in\R^p$, $BB^{\dag}x\in\im(B)\subset\im(A)$,
	then $AA^{\dag}(BB^{\dag}x)=BB^{\dag}x$. Hence $AA^{\dag}BB^{\dag}=BB^{\dag}$.
	On the other hand,
	from the definition of Moore-Penrose inverses,
	$BB^{\dag}=(BB^{\dag})^T=(AA^{\dag}BB^{\dag})^T
	=(BB^{\dag})^T(AA^{\dag})^T=
	BB^{\dag}AA^{\dag}$.
\end{proof}

\section{Finite games and the semitensor product of matrices}
\label{sec:ERofFGS:preliminary}

%
%

A noncooperative  strategic form finite game can be described as a triple $(N,S,c)$, where
\begin{enumerate}
\item $N=\{1,\dots,n\}$ is the set of players;
\item $S^i=\{1,\dots,k_i\}$ is the set of strategies of player $i$, $i=1,\dots,n$;
$S=\prod_{i=1}^nS^i$ is the set of strategy profiles;
\item $c=\{c_1,c_2,\dots,c_n\}$, where $c_i:~S\to \R$ is the payoff function of player $i$,
	$i=1,\dots,n$.
\end{enumerate}

Hereinafter $S^{-i}$ denotes $\prod_{j=1,j\ne i}^{n}S^j$, and similarly for a strategy profile $s=(s_1,\dots,s_n)
\in S$, $s^{-i}$ denotes $(s_1,\dots,s_{i-1},s_{i+1},\dots,s_n)$.

A pure Nash equilibrium (cf. \cite{Monderer1996PotentialGames}, etc.)
is a strategy profile $(s_1^*,\dots,s_n^*)\in S$ such that
for all $i\in N$, all $s_i\in S^i$, $c_i(s_1^*,\dots,s_{i-1}^*,s_i,s_{i+1}^*,\dots,s_n^*)\le
c_i(s_1^*,\dots,s_n^*)$. That is, if the players play a pure Nash equilibrium,
none of the players will improve their payoffs if only one player changes his/her strategy.
Note that a finite game $(N,S,c)$ may have no pure Nash equilibria.
However, if mixed strategies are considered, every finite game has a mixed strategy Nash equilibrium
\cite{Nash1950EquilibriumPoint}. As usual, for each player $i\in N$, a mixed strategy $x_i$ is a probability
distribution $S^i\to \R$, i.e., for all $s\in S^i$, $x_i(s)\ge0$ and $\sum_{s\in S^i}x_i(s)=1$,
which means that player $i$ plays strategy $s$ with probability $x_i(s)$.
For each $i\in N$, the set of mixed strategies is denoted by $\Dt S^i$.
Then the product probability distribution is $\prod_{i=1}^n\Dt S^i$.
In particular,
a mixed strategy profile $(x_1^*,\dots,x_n^*)\in\prod_{i=1}^{n}{\Dt S^i}$ is called uniformly
mixed strategy profile, if
for all $i\in N$, all $s\in S^i$, $x_i^*(s)=\frac{1}{k_i}$.
Under the product probability distribution, payoffs $c_i$ of players $i\in N$ are naturally
extended to the expected payoffs over all pure strategy profiles
$c_i^*:\prod_{j=1}^n\Dt S^j\to\R,x=(x_1,\dots,x_n)\mapsto \sum_{s\in S}c_i(s) x(s)
=\sum_{s\in S}c_i(s)\prod_{j=1}^n x_j(s)$
that is linear in each $x_i$. Hereinafter we will use $c_i$ to denote $c_i^*$ since no confusion will occur.
A mixed Nash equilibrium is a mixed strategy profile $x^*=(x^*_1,\dots,x_n^*)\in\prod_{i=1}^n\Dt S^i$
such that for all players $i\in N$, all mixed strategies $x_i\in \Dt S^i$,
$c_i(x^*)\ge c_i(x^*_1,\dots,x^*_{i-1},x_i,x^*_{i+1},\dots,x^*_{n})$.
Note that $x^*\in\prod_{i=1}^n\Dt S^i$ is a mixed strategy Nash equilibrium iff,
for all players $i\in N$, all pure strategies $s_i\in S^i$,
$c_i(x^*)\ge c_i(x^*_1,\dots,x^*_{i-1},s_i,x^*_{i+1},\dots,x^*_{n})$.


Next we introduce the vector space structure of finite games based on the STP
of matrices built in \cite{Cheng2014PotentialGame}.
In this framework, mixed strategies and mixed strategy profiles 
can be represented as real vectors, hence payoffs can be represented as linear mapping forms.
Then all properties
of finite games can be revealed intuitively from matrices.

\begin{definition}\cite{Cheng(2011book)}
	Let $A\in \R_{m\times n}$, $B\in \R
	_{p\times q}$, and $\alpha=\mbox{lcm}
	(n,p)$ be the least common multiple of $n$ and $p$. The STP of $A$
	and $B$ is defined as
	\begin{equation*}
		A\ltimes B = (A\otimes I_{\frac{\alpha}{n}})(B\otimes I_{\frac
		{\alpha}{p}}),
		\label{def_of_stp}
	\end{equation*}
	where $\otimes$ denotes the Kronecker product.
\end{definition}

The STP of matrices is a generalization of conventional matrix product,
and many properties of the conventional matrix product remain valid,
e.g., associative law, distributive law, reverse-order law ($(A\ltimes B)^T=B^T\ltimes A^T$), etc..
Besides, for all $x\in\R^{t}$ and $A\in\R_{m\times n}$, it holds that $x\ltimes A=(I_t\otimes A)x$.
Throughout this paper, the default matrix product is STP,
so the product of two arbitrary matrices is well defined, and the symbol $\ltimes$ is usually omitted.


To use matrix expression to games, we identify
$$
j\sim \dt_{k_i}^j,\quad j=1,\dots,k_i,
$$
then $S^i\sim \D_{k_i}$, $i=1,\dots,n$. It follows that the payoff functions can be expressed as
\begin{align}\label{1.2}
c_i(x_1,\dots,x_n)=V^c_i\ltimes_{j=1}^nx_j,\quad i=1,\dots,n,
\end{align}
where $(V^c_i)^T\in \R^{k}$ is called the structure vector of $c_i$, $x_j\in\Dt_{k_j}$,
$j=1,\dots,n$, hereinafter  $k:=\prod_{i=1}^nk_i$.
Define the structure vector of a game $G$ by
\begin{align}\label{1.201}
(V^c_G)^T=(V^c_1,V^c_2,\dots,V^c_n)^T\in \R^{nk}.
\end{align}
Then it is clear that the set of strategic form finite games with $|N|=n$, and  $|S^i|=k_i$, $i=1,\dots,n$,
denoted by ${\mathcal G}_{[n;k_1,\dots,k_n]}$, has a natural vector space structure as
\begin{align}\label{1.202}
{\mathcal G}_{[n;k_1,\dots,k_n]}\sim \R^{nk}.
\end{align}
For a given game $G\in {\mathcal G}_{[n;k_1,\dots,k_n]}$, its structure vector $V^c_G$ completely determines $G$. So the vector space structure (\ref{1.202}) is very natural and reasonable.

%
%

\begin{remark}
	Consider the Euclidean space $\R^{nk}$ with the weighted inner product: for all $x,y\in\R^{nk}$,
$\langle x,y\rangle=x^TQy$, where $Q\in\R_{nk\times nk}$ is a positive definite symmetric
matrix. Then for a linear operator (i.e., a matrix in $\R_{nk\times p}$) $A:\R^{nk}\to\R^p$,
the orthogonal projection onto $\im(A)$ is $A(A^TQA)^{\sharp}A^TQ$.
In this paper, we consider the convetional inner product, that is, the case that $Q=I_{nk}$.
While the inner product considered in \cite{Candogan2011FlowsDecompositionofGames}
is with
$$
Q=\diag\left(\underbrace{k_1,\dots,k_1}_{k},\underbrace{k_2,\dots,k_2}_{k},\dots,\underbrace{k_n,\dots,k_n}_{k}\right).
$$
Hence the (pure) harmonic games considered in this paper are a little bit different from the ones
considered in \cite{Candogan2011FlowsDecompositionofGames}.
Despite of this difference, we can still prove that Harmonic games of this paper also have the
uniformly mixed strategy profiles as mixed Nash equilibria.
\end{remark}

\section{Orthogonal projections onto subspaces of finite games}
\label{sec:ProjontoGameSubsapce}

Based on the above preliminaries, we are ready to show the orthogonal projections
onto subspaces of finite games in the framework of STP.


\subsection{Subspaces of nonstrategic games}

Let us define some notations. Part of these notations 
for the first time appear in \cite{Cheng2014PotentialGame}.

Define

\begin{equation}
	\begin{split}
		k^{[p,q]}:=\left\{
		\begin{array}[]{ll}
			\prod_{j=p}^{q}k_j, &\text{ if } q\ge p,\\
			1, &\text{ if } q<p,
		\end{array}\right.
	\end{split}
	\label{eqn:game_kpq}
\end{equation}

\begin{equation}
	\begin{split}
		E_i :=& I_{k^{[1,i-1]}}\otimes {\bf1}_{k_i}\otimes I_{k^{[i+1,n]}}
		\in\R_{k\times \frac{k}{k_i}},\\
		{\bf e}_i :=& E_iE_i^T=I_{k^{[1,i-1]}}\otimes {\bf1}_{k_i\times
		k_i}\otimes I_{k^{[i+1,n]}}
		\in\R_{k\times k}, \\& i\in[1,n].
	\end{split}
	\label{eqn:game_E_i}
\end{equation}

It is easy to verify the following proposition.

\begin{proposition}\label{prop_EMPinverse}
	For all $i\in[1,n]$,
	\begin{equation}
		E_i^{\dag}=\frac{1}{k_i}E_i^T,\quad E_iE_i^{\dag}=\frac{1}{k_i}{\bf e}_i.
	\end{equation}
\end{proposition}

Define

\begin{equation}
	\begin{split}
		B_N:=E_1\oplus\cdots\oplus E_n\in\R_{(nk)\times(\sum_{i=1}^{n}\frac
		{k}{k_i})},
	\end{split}
	\label{eqn:game_B_N}
\end{equation}
then the following proposition holds.
\begin{proposition}\label{prop_BNMPinverse}
	\begin{equation}
		B_N^{\dag}=E_1^{\dag}\oplus\cdots\oplus E_n^{\dag},\quad B_NB_N^{\dag}=
\frac{1}{k_1}{\bf e}_1\oplus\cdots\oplus\frac{1}{k_n}{\bf e}_n.
	\end{equation}
\end{proposition}

It follows directly that each $E_i$ is of full column rank, hence so is $B_N$, i.e.,
	\begin{equation}
		\rank(B_N)=\sum_{i=1}^{n}\frac{k}{k_i}.
		\label{eqn:rankBN}
	\end{equation}

%
%

From the results in \cite{Candogan2011FlowsDecompositionofGames}, nonstrategic games
are exactly the games such that the payoff of each player does not depend on the strategy played
by the player himself/herself.
Then their formal definition is obtained as below.

\begin{theorem}\label{thm7:game_decomposition}
	The nonstrategic games are exactly the games $(N,S,c)$ in $\Gr_{[n;k_1,\dots,k_n]}$
	satisfying that
	\begin{equation}\begin{split}
	 	&\forall i\in[1,n], \forall y\in S^i,
		\forall s\in S^{-i}, \\&
		\frac{1}{k_i}\sum_{x\in S^i}c_i(x,s)-
		c_i(y,s)=0.
		\end{split}
		\label{eqn:game_Nonstrategicgame}
	\end{equation}
\end{theorem}

In the framework of STP, Theorem \ref{thm7:game_decomposition} can be represented as
the following Theorem \ref{thm12:game_decomposition}.

\begin{theorem}\label{thm12:game_decomposition}
	Consider the finite game space $\Gr_{[n;k_1,\dots,k_n]}$.
	The nonstrategic subspace is
	\begin{equation}
		\begin{split}
			\mathcal{N} &= \im(B_N) = \im(E_1\oplus\cdots\oplus E_n)\\
			&= \im(B_NB_N^{\dag}) = \im(E_1E_1^{\dag}\oplus\cdots\oplus
			E_nE_n^{\dag}) \\
			&= \im\left(\frac{1}{k_1}{\bf e}_1\oplus\cdots\oplus
			\frac{1}{k_n}{\bf e}_n\right),
		\end{split}
		\label{}
	\end{equation}
	and satisfies
	\begin{equation}
		\dim(\Na)=\sum_{i=1}^n{\frac{k}{k_i}}.
		\label{eqn:dim_Na}
	\end{equation}
\end{theorem}

\begin{proof}
	By Propositions \ref{prop_linearspace} and \ref{prop_BNMPinverse}, we only need to prove
	$\Na=\im\left(\frac{1}{k_1}{\bf e}_1\oplus\cdots\oplus\frac{1}{k_n}{\bf e}_n\right)$. Then
	\eqref{eqn:dim_Na} naturally holds by \eqref{eqn:rankBN}.
	It is easy to verify that for all $i\in[1,n]$, $\frac{1}{k_i}{\bf e}_i$ are symmetric,
	and $(\frac{1}{k_i}{\bf e}_i)^2=\frac{1}{k_i}{\bf e}_i$,
	hence $\left(\frac{1}{k_1}{\bf e}_1\oplus\cdots\oplus\frac{1}{k_n}{\bf e}_n\right)$ is an orthogonal projection.


	Arbitrarily choose a game
	$u=(u_1,\dots,u_n)^T\in\Gr_{[n;k_1,\dots,k_n]}$, where each $u_i^T$ belongs to $\R^{k}$.
	From \eqref{eqn:game_Nonstrategicgame},
	one has $u\in\Na$ iff, for all $i\in[1,n]$, all $X\in\Dt_{k^{[1,i-1]}}$, all $y\in\Dt_{k_i}$,
	all $Y\in\Dt_{k^{[i+1,n]}}$, $\frac{1}{k_i}\sum_{z\in\Dt_{k_i}}u_iXzY=\frac{1}{k_i}u_iX(\sum_{z\in\Dt_{k_i}}z)Y
	=\frac{1}{k_i}u_iX{\bf1}_{k_i}Y=\frac{1}{k_i}u_i(X\otimes{\bf1}_{k_i}\otimes Y)=u_iXyY=u_i(X\otimes y\otimes Y)$
	iff, for all $i\in[1,n]$, $\frac{1}{k_i}u_i(I_{k^{[1,i-1]}}\otimes {\bf1}_{k_i\times k_i}\otimes
	I_{k^{[i+1,n]}})=u_i\left(\frac{1}{k_i}{\bf e}_{i}\right)=u_i(I_{k^{[1,i-1]}}\otimes {I}_{k_i}\otimes
	I_{k^{[i+1,n]}})=u_i$ iff, $u\in\\\im\left(\frac{1}{k_1}{\bf e}_1\oplus\cdots\oplus\frac{1}{k_n}{\bf e}_n\right)$.
	Hence $\Na=\im(B_N)$.

\end{proof}

Next we characterize the projection of a game onto $\Na$.
By Theorem \ref{thm12:game_decomposition}, the following theorem holds.
\begin{theorem}\label{thm19:game_decomposition}
	The projection of a game $(N,S,c)\in\Gr_{[n;k_1,\dots,k_n]}$ onto the nonstrategic subspace
	$\Na$ is $(N,S,c')$, where $c=\{c_1,\dots,c_n\}$, $c'=\{c_1',\dots,c_n'\}$,
	for all $i\in[1,n]$, all $x\in S^i$, all $y\in S^{-i}$, $c_i'(x,y)=\frac{1}{k_i}\sum_{z\in S^i}c_i(z,y)$.
\end{theorem}

\begin{proof}
	Consider the vector form of game $(N,S,c)$, i.e., $u^T=(u_1,\dots,u_n)^T\in\R^{nk}$, where each
	$u_i^T$ belongs to $\R^n$. Then the projection of $u^T$ onto $\Na$, i.e., $(\frac{1}{k_1}u_1{\bf e}_{1},
	\dots,\frac{1}{k_n}u_n{\bf e}_{n})^T$ satisfies that for all $i\in[1,n]$, all $X=\ltimes_{j=1}^nX_j\in\Dt_{k}$,
	$X^{[1,i-1]}:=\ltimes_{j=1}^{i-1}X_j$, $X^{[j+1,n]}:=\ltimes_{j=i+1}^{n}X_j$,
	\begin{equation}
		\begin{split}
			&\frac{1}{k_i}u_i{\bf e}_{i} X\\
			=&\frac{1}{k_i}u_i\left(I_{k^{[1,i-1]}}\otimes{\bf1}_{k_i\times k_i}\otimes I_{k^{[i+1,n]}}\right)\\
			&\left(X^{[1,i-1]}\ltimes X_i\ltimes X^{[i+1,n]}\right)\\
			=&\frac{1}{k_i}u_i\left(I_{k^{[1,i-1]}}\otimes{\bf1}_{k_i\times k_i}\otimes I_{k^{[i+1,n]}}\right)\\
			&\left(X^{[1,i-1]}\otimes X_i\otimes X^{[i+1,n]}\right)\\
			=&\frac{1}{k_i}u_i\left(\left(I_{k^{[1,i-1]}}X^{[1,i-1]}\right)\otimes\left({\bf1}_{k_i\times k_i}X_i\right)\otimes\right.\\
			&\left. \left(I_{k^{[i+1,n]}}X^{[i+1,n]}\right)\right)\\
			=&\frac{1}{k_i}u_i \left( X^{[1,i-1]} \otimes {\bf1}_{k_i}
			\otimes X^{[i+1,n]}\right)\\
			=&\frac{1}{k_i}u_i \left( X^{[1,i-1]} \ltimes \left(\sum_{z\in\Dt_{k_i}}z\right)
			\ltimes X^{[i+1,n]}\right)\\
			=&\frac{1}{k_i}\sum_{z\in\Dt_{k_i}}u_i \left( X^{[1,i-1]} \ltimes z
			\ltimes X^{[i+1,n]}\right).
		\end{split}
		\label{}
	\end{equation}
	That is, the conclusion holds.
\end{proof}

\subsection{Subspaces of potential games}

In \cite{Monderer1996PotentialGames}\label{thm11:game_decomposition},
potential games are defined as the games $(N,S,c)$
in $\Gr_{[n;k_1,\dots,k_n]}$ satisfying
	\begin{equation}
		\begin{split}
			&\exists \phi:S\to\R, \forall i\in[1,n],
			\forall x,y\in S^i,\forall z\in S^{-i},\\
			&c_i(x,z)-c_i(y,z)=\phi(x,z)-\phi(y,z),
		\end{split}
		\label{eqn:game_Potentialgames}
	\end{equation}
where $\phi$ is called potential function.

From this definition, nonstrategic games are exactly the potential games that have constant potential functions.
For every nonstrategic game, every strategy profile is a pure Nash equilibrium.


Necessary notations are given as follows.
Regard $2^{[1,n]}$ as an index set, for all $N_s\subset [1,n]$,
\begin{equation}
	\begin{split}
		{\bf e}_{N_s} :=\left\{
		\begin{array}[]{ll}
			\prod_{i\in N_s}{\bf e}_i, &\text{ if } N_s\ne\emptyset,\\
			I_k, & \text{otherwise}.
		\end{array}\right.
	\end{split}
	\label{eqn:game_eNs}
\end{equation}

Then
\begin{equation}
	\begin{split}
		{\bf e}_{N_s}=A_1\otimes A_2\otimes\cdots\otimes A_n,
	\end{split}
	\label{}
\end{equation}
where
\begin{equation}
	A_i = \left\{
	\begin{array}[]{ll}
		I_{k_i}, &\text{ if }i\notin N_s,\\
		{\bf1}_{k_i\times k_i}, &\text{ if }i\in N_s.
	\end{array}
	\right.
	\label{}
\end{equation}

Given a finite game $(V_G^c)^T=(V_1^c,\dots,V_n^c)^T\in\R^{nk}$,
where each $(V_i^c)^T$ belongs to $\R^k$,
the corresponding potential equation constructed in \cite{Cheng2014PotentialGame} is
\begin{equation}
	\begin{bmatrix}
		-E_1 & E_2 & 0 & \cdots & 0\\
		-E_1 & 0 & E_3 & \cdots & 0\\
		\vdots &\vdots &\vdots & \ddots & \vdots \\
		-E_1 & 0 & 0 & \cdots & E_n
	\end{bmatrix}\begin{bmatrix}
		\xi_1\\ \vdots\\ \xi_n
	\end{bmatrix}=\begin{bmatrix}
		(V_2^c)^T-(V_1^c)^T \\\vdots\\ (V_n^c)^T-(V_1^c)^T
	\end{bmatrix},
	\label{eqn:potentialequation}
\end{equation}
where $E_1,\dots,E_n$ are defined in \eqref{eqn:game_E_i},
for each $i\in[1,n]$, $\xi_i\in\R^{\frac{k}{k_i}}$.
It is proved in \cite{Cheng2014PotentialGame} that game $V_G^c$ is potential iff Eqn.
\eqref{eqn:potentialequation} has solution, and if Eqn. \eqref{eqn:potentialequation}
has a solution $\xi_1^*,\dots,\xi_n^*$, then $V_1^c-(\xi_1^*)^TE_1^T$ is the structure
vector of a potential function of $V_G^c$.

Putting a free vector $y\in\R^k$ into Eqn. \eqref{eqn:potentialequation}, Eqn. \eqref{eqn:potentialequation}
is equivalent to
\begin{equation}
	\begin{bmatrix}
		I_k & 0 & 0 & 0 & \cdots & 0\\
		0 & -E_1 & E_2 & 0 & \cdots & 0\\
		0 & -E_1 & 0 & E_3 & \cdots & 0\\
		\vdots & \vdots &\vdots &\vdots & \ddots & \vdots \\
		0 & -E_1 & 0 & 0 & \cdots & E_n
	\end{bmatrix}\begin{bmatrix}
		y\\\xi_1\\ \vdots\\ \xi_n
	\end{bmatrix}=\begin{bmatrix}
		y\\(V_2^c)^T-(V_1^c)^T \\\vdots\\ (V_n^c)^T-(V_1^c)^T
	\end{bmatrix},
	\label{eqn:potentialequation1}
\end{equation}
equivalent to
\begin{equation}
	\begin{bmatrix}
		I_k & 0 & 0 & 0 & \cdots & 0\\
		I_k & -E_1 & E_2 & 0 & \cdots & 0\\
		I_k & -E_1 & 0 & E_3 & \cdots & 0\\
		\vdots & \vdots &\vdots &\vdots & \ddots & \vdots \\
		I_k & -E_1 & 0 & 0 & \cdots & E_n
	\end{bmatrix}\begin{bmatrix}
		y\\\xi_1\\ \vdots\\ \xi_n
	\end{bmatrix}=\begin{bmatrix}
		y\\y+(V_2^c)^T-(V_1^c)^T \\\vdots\\y+ (V_n^c)^T-(V_1^c)^T
	\end{bmatrix},
	\label{eqn:potentialequation2}
\end{equation}
and equivalent to
\begin{equation}\begin{split}
	&\begin{bmatrix}
		I_k & E_1  & 0 & 0 & \cdots & 0\\
		I_k & 0 & E_2 & 0 & \cdots & 0\\
		I_k & 0 & 0 & E_3 & \cdots & 0\\
		\vdots & \vdots &\vdots &\vdots & \ddots & \vdots \\
		I_k & 0 & 0 & 0 & \cdots & E_n
	\end{bmatrix}\begin{bmatrix}
		y-E_1\xi_1\\\xi_1\\ \vdots\\ \xi_n
	\end{bmatrix}\\
	=&\begin{bmatrix}
		y\\y+(V_2^c)^T-(V_1^c)^T \\\vdots\\y+ (V_n^c)^T-(V_1^c)^T
	\end{bmatrix}.
	\end{split}
	\label{eqn:potentialequation3}
\end{equation}

Substituting $y=(V_1^c)^T$ into Eqn. \eqref{eqn:potentialequation3}, one has
\begin{equation}
	\begin{bmatrix}
		I_k & E_1  & 0 & 0 & \cdots & 0\\
		I_k & 0 & E_2 & 0 & \cdots & 0\\
		I_k & 0 & 0 & E_3 & \cdots & 0\\
		\vdots & \vdots &\vdots &\vdots & \ddots & \vdots \\
		I_k & 0 & 0 & 0 & \cdots & E_n
	\end{bmatrix}\begin{bmatrix}
		(V_1^c)^T-E_1\xi_1\\\xi_1\\ \vdots\\ \xi_n
	\end{bmatrix}=\begin{bmatrix}
		(V_1^c)^T\\(V_2^c)^T \\\vdots\\ (V_n^c)^T
	\end{bmatrix},
	\label{eqn:potentialequation4}
\end{equation}
hereinafter
\begin{equation}
	B_P:=\begin{bmatrix}
		I_k & E_1  & 0 & 0 & \cdots & 0\\
		I_k & 0 & E_2 & 0 & \cdots & 0\\
		I_k & 0 & 0 & E_3 & \cdots & 0\\
		\vdots & \vdots &\vdots &\vdots & \ddots & \vdots \\
		I_k & 0 & 0 & 0 & \cdots & E_n
	\end{bmatrix}
	\in\R_{(nk)\times (k+\sum_{i=1}^{n}\frac{k}{k_i})}.
	\label{eqn:game_Bp}
\end{equation}

Then from Eqn. \eqref{eqn:potentialequation4}, a game $V_G^c$ is potential
iff $(V_G^c)^T$ belongs to $\im(B_P)$. \cite[Lemma 22]{Cheng2014PotentialGame}
shows that
$\rank\begin{bmatrix}
		-E_1 & E_2 & 0 & \cdots & 0\\
		-E_1 & 0 & E_3 & \cdots & 0\\
		\vdots &\vdots &\vdots & \ddots & \vdots \\
		-E_1 & 0 & 0 & \cdots & E_n
\end{bmatrix}=\sum_{i=1}^n\frac{k}{k_i}-1$.
Then the following theorem holds.

\begin{theorem}\label{thm:game_Potentialgamesubspace}
	Consider the finite game space $\Gr_{[n;k_1,\dots,k_n]}$.
	The potential subspace is $\Gr_P=\im(B_P)$,
	and satisfies that
	$\dim(\Gr_P)=k+\sum_{i=1}^{n}\frac{k}{k_i}-1$.
\end{theorem}

By Theorem \ref{thm:game_Potentialgamesubspace} and Proposition \ref{prop_linearspace},
$B_P(B_P)^{\dag}$ is the orthogonal projection onto $\Gr_P$.
Later we will give polynomial representation for $B_P(B_P)^{\dag}$. This part is much more difficult
than the case of nonstrategic games, and we will finish it in next section.
For a potential game $V_G^c$, using Eqn. \eqref{eqn:potentialequation4},
in next section we will also find polynomial representation for its potential functions.
Before proceeding it, we turn to the study of pure harmonic games.

\subsection{Subspaces of pure harmonic games}

The following theorem follows from Theorem \ref{thm:game_Potentialgamesubspace}.

\begin{theorem}\label{prop_game_Purepotentialsubspace}
	Consider the finite game space $\Gr_{[n;k_1,\dots,k_n]}$.
	The pure harmonic subspace is $\mathcal{H}=\im(B_P)^{\bot}$,
	and satisfies $\dim(\mathcal{H})=(n-1)k-\sum_{i=1}^{n}\frac{k}{k_i}+1$.
\end{theorem}

By Proposition \ref{prop_linearspace}, the following Proposition
\ref{prop_game_Pureharmonicsubspace} holds.

\begin{proposition}\label{prop_game_Pureharmonicsubspace}
	$\im(B_P)=\im(B_P(B_P)^{\dag})$, $\im(B_P)^{\bot}=\ker((B_P)^T)=\im(I_{nk}-B_PB_P^{\dag})$.
\end{proposition}

Then from Theorem \ref{prop_game_Purepotentialsubspace},
Propositions \ref{prop_linearspace} and \ref{prop_game_Pureharmonicsubspace},
the following theorem follows.

\begin{theorem}\label{thm5:game_decomposition}
	\begin{equation}
		\begin{split}
			\Ha =& \ker({\bf1}_n^T\otimes I_k)\cap
			\ker\left(\frac{1}{k_1}{\bf e}_1\oplus\cdots\oplus
			\frac{1}{k_n}{\bf e}_n\right)\\
			=& \im\left(I_{nk}-\frac{1}{n}\left(
			{\bf1}_{n\times n}\otimes I_k\right)\right) \cap \\&
			\im\left(\left(I_k-\frac{1}{k_1}{\bf e}_1\right)\oplus\cdots\oplus
			\left(I_k-\frac{1}{k_n}{\bf e}_n\right)\right),\\
		\end{split}
		\label{eqn:game_decomposition2}
	\end{equation}
	where ${\bf e}_1,\dots,{\bf e}_n$ are defined in \eqref{eqn:game_E_i}.
\end{theorem}

From Theorem \ref{thm5:game_decomposition}, Theorem
\ref{thm6:game_decomposition} 
which follows shows a direct definition for pure harmonic games.

\begin{theorem}\label{thm6:game_decomposition}
	The pure harmonic games are exactly the games $(N,S,c)$
	in $\Gr_{[n;k_1,\dots,k_n]}$
	satisfying the following two conditions
	\begin{eqnarray}
			\forall s\in S,\quad \sum_{i=1}^n{c_i(s)=0},
			\label{eqn:game_Pureharmonicgame1}\\
			\forall i\in[1,n], \forall y\in S^{-i},\quad
			\sum_{x\in S^i}c_i(x,y)=0.\label{eqn:game_Pureharmonicgame4}
	\end{eqnarray}
\end{theorem}

\begin{proof}
	We use Theorem \ref{thm5:game_decomposition} to prove it.
	Arbitrarily choose a game
	$u=(u_1,\dots,u_n)^T\in\Gr_{[n;k_1,\dots,k_n]}$, where each $u_i^T$ belongs to $\R^{k}$.

	$u\in\ker({\bf1}_n^T\otimes I_k)$ iff, $\sum_{i=1}^nu_i^T=0$ iff, \eqref{eqn:game_Pureharmonicgame1} holds.

	$u\in\ker\left(\frac{1}{k_1}{\bf e}_1\oplus\cdots\oplus\frac{1}{k_n}{\bf e}_n\right)$ iff,
	for all $i\in[1,n]$, ${\bf e}_iu_i^T=E_iE_i^Tu_i^T=0$ iff, for all $i\in[1,n]$, $u_iE_i=0$ iff,
	for all $i\in[1,n]$, $0=u_i(I_{k^{[1,i-1]}}\otimes{\bf1}_{k_i}\otimes I_{k^{[i+1,n]}})$ iff,
	for all $i\in[1,n]$, all $X\in\Dt_{k^{[1,i-1]}}$, all $Y\in\Dt_{k^{[i+1,n]}}$, $u_i(X\otimes{\bf1}_{k_i}\otimes Y)=
	u_i(X\ltimes{\bf1}_{k_i}\ltimes Y)=
	u_iX(\sum_{y\in\Dt_{k_i}}y)Y=\sum_{y\in\Dt_{k_i}}u_iXyY=0$ iff, \eqref{eqn:game_Pureharmonicgame4} holds.
%
\end{proof}

Note that finite games satisfying \eqref{eqn:game_Pureharmonicgame4} are called
{\it normalized} in \cite{Candogan2011FlowsDecompositionofGames},
and the subspace of normalized games is the orthogonal complement of
the nonstrategic subspace.

\cite[Proposition 5.1]{Candogan2011FlowsDecompositionofGames} shows that ``Harmonic
 games generically do not
have pure Nash equilibria.''
Next we prove a stronger result,
i.e., the following Theorem \ref{thm16:game_decomposition}. Note that the proofs
of \cite[Proposition 5.1]{Candogan2011FlowsDecompositionofGames} and the following
Theorem \ref{thm16:game_decomposition} are essentially different.
The following Proposition \ref{thm13:game_decomposition} will be used in Theorem
\ref{thm16:game_decomposition}.

\begin{proposition}
	Consider the finite game space $\Gr_{[n;k_1,\dots,k_n]}$.
	A pure harmonic game $(N,S,c)\in\Ha$ possesses a pure Nash equilibrium
	$s^*=(s_1^*,\dots,s_n^*)\in S$ iff
	for all $i\in[1,n]$, all $s\in S^i$, $c_i(s_1^*,\dots,s_{i-1}^*,s,s_{i+1}^*,\dots,s_n^*)=0$.
	\label{thm13:game_decomposition}
\end{proposition}

\begin{proof}
	The ``if'' part holds naturally. Next we prove the ``only if'' part.
	
	From the definition of pure Nash equilibrium, for all $i\in[1,n]$, all $s\in S^i$,
	$c_i(s_1^*,\dots,s_{i-1}^*,s,s_{i+1}^*,\dots,s_n^*)\le c_i(s^*)$.
	Then from \eqref{eqn:game_Pureharmonicgame4}, for all $i\in[1,n]$,
	$0=\sum_{t\in S^i}c_i(s_1^*,\dots,\\s_{i-1}^*,t,s_{i+1}^*,\dots,s_n^*)\le k_i c_i(s^*)$.
	From \eqref{eqn:game_Pureharmonicgame1}, for all $i\in[1,n]$, $c_i(s^*)=0$.
	Then for all $i\in[1,n]$, all $s\in S^{i}$, $c_i(s_1^*,\dots,s_{i-1}^*,s,s_{i+1}^*,\dots,s_n^*)\le 0$.
	Again from \eqref{eqn:game_Pureharmonicgame4},
	for all $i\in[1,n]$, all $s\in S^{i}$, $c_i(s_1^*,\dots,s_{i-1}^*,s,s_{i+1}^*,\dots,s_n^*)= 0$.
\end{proof}


\begin{theorem}
	Pure harmonic games generically do not have pure Nash equilibria.
	\label{thm16:game_decomposition}
\end{theorem}

\begin{proof}
	Consider the finite game space $\Gr_{[n;k_1,\dots,k_n]}$. For each strategy profile $s\in S$,
	denote $\Gr_{s}\subset \Ha$ by the set of pure harmonic games that have $s$ as a pure Nash equilibrium.
	By Proposition \ref{thm13:game_decomposition}, for every two games $(N,S,c'),(N,S,c'')$ both in $\Gr_s$,
	$(N,S,c'-c'')$ is also in $\Gr_s$. Then for each $s\in S$,
	$\Gr_s$ is a subspace of $\Ha$. Next we prove for each $s\in S$, $\dim(\Gr_s)
	<\dim(\Ha)$, then $\cup_{s\in S}\Gr_s$ is a measure $0$ subset of $\Ha$. That is, pure harmonic
	games generically do not have pure Nash equilibria.

	Without loss of generality, we assume that $s=\dt_{k}^1\in\Dt_{k}$.
	Choose an arbitrary pure harmonic game $u^T=(u_1,\dots,u_n)^T\in\Ha$, where each $u_i^T$ belongs to
	$\R^k$. Then by Proposition \ref{thm13:game_decomposition},
	$u^T$ has a pure Nash equilibrium $s$ iff, for all $i\in[1,n]$, all $x\in \Dt_{k_i}$,
	$0=u_i\dt_{k^{[1,i-1]}}^1 x \dt_{k^{[i+1,n]}}^1=u_i(\dt_{k^{[1,i-1]}}^1\otimes x \otimes \dt_{k^{[i+1,n]}}^1)$
	iff, for all $i\in[1,n]$, $0=u_i(\dt_{k^{[1,i-1]}}^1\otimes I_{k_i} \otimes \dt_{k^{[i+1,n]}}^1)$.
	Then by Theorem \ref{thm5:game_decomposition},
	\begin{equation}
		\begin{split}
			\Gr_{\dt_k^1} = \ker(G_{\dt_k^1}),
		\end{split}
		\label{}
	\end{equation}
	where  $$G_{\dt_k^1}:=\begin{bmatrix}
				I_k & \cdots & I_k\\
				E_1^T \\
				 & \ddots\\
				 &  & E_n^T\\
				F_1^T\\
				 & \ddots\\
				 & & F_n^T
			 \end{bmatrix}\in\R_{(k+\sum_{i=1}^k{\frac{k}{k_i}}+\sum_{i=1}^n{k_i})\times(nk)},$$
	$E_1,\dots,E_n$ are defined in \eqref{eqn:game_E_i},
	$F_i^T=(\dt_{k^{[1,i-1]}}^1)^T \otimes I_{k_i} \otimes (\dt_{k^{[i+1,n]}}^1)^T$, $i\in[1,n]$.

	Note that each of the last $\sum_{i=1}^n{k_i}$ rows of $G_{\dt_k^1}$ correspondes to a game
	that is not potential. Then $\rank(G_{\dt_k^1})>\rank(B_P)=\dim(\Po\oplus\Na)$.
	Hence $\dim(\Gr_{\dt_k^1})=nk-\rank(G_{\dt_k^1})<\dim(\Ha)=nk-\dim(\Po\oplus\Na)$.
\end{proof}

Theorem \ref{thm16:game_decomposition} naturally yields the following corollary.

\begin{corollary}\label{cor1:game_decomposition}
	A pure harmonic game in $\Gr_{[2;2,2]}$ has a pure Nash equilibrium iff it is trivial.
\end{corollary}

\begin{proof}
	For this case, for all $s\in\{1,2\}\times\{1,2\}$, $\rank(G_{s})=8$, then
	$\dim(\Gr_{s})=8-\rank(G_{s})=0$. That is, the conclusion holds.
\end{proof}

\subsection{Subspaces of pure potential games}

Define
\begin{equation}
	P_N := \begin{bmatrix}
			I_k-\frac{1}{k_1}{\bf e}_1 \\ \vdots \\
			I_k-\frac{1}{k_n}{\bf e}_n
		\end{bmatrix}\in\R_{(nk)\times k},
	\label{eqn:PN}
\end{equation}
where ${\bf e}_1,\dots,{\bf e}_n$ are defined in \eqref{eqn:game_E_i}.

It follows that
\begin{eqnarray}
	P_N^TB_N=0,\label{eqn:game_Purepotential2}\\
	\begin{bmatrix}
	P_N,B_N
\end{bmatrix}=B_P\begin{bmatrix}
	I_k & \\
	-\frac{1}{k_1}E_1^T & I_{\frac{k}{k_1}} \\
	\vdots && \ddots\\
	-\frac{1}{k_n}E_n^T & && I_{\frac{k}{k_n}} \\
\end{bmatrix},
	\label{eqn:game_Purepotential3}
\end{eqnarray}
where $B_N$ is defined in \eqref{eqn:game_B_N},
and $E_1,\dots,E_n$ are defined in \eqref{eqn:game_E_i}.

In view of \eqref{eqn:game_Purepotential2} and
\eqref{eqn:game_Purepotential3}, the following theorem holds.
\begin{theorem}
	Consider the finite game space $\Gr_{[n;k_1,\dots,k_n]}$.
	The pure potential subspace is $\Po=\im(P_N)=\im(P_NP_N^{\dag})$, satisfies
	$\dim(\Po)=k-1$, and every $k-1$ distinct columns of $P_N$ is a basis of $\Po$.
	\label{thm18:game_decomposition}
\end{theorem}

\begin{proof}
	Eqn. \eqref{eqn:game_Purepotential2} implies that $\im(P_N)\bot\Na$.
	Eqn. \eqref{eqn:game_Purepotential3} implies that $\im(P_N)\oplus\Na=\im(B_P)=\Gr_P$.
	Then $\im(P_N)=\Po$, and $\dim(\Po)=\dim(\Gr_P)-\dim(\Na)=k-1$.
	Note that the sum of all columns of $P_N$ are the zero vector of $\R^{nk}$, then
	every distinct $k-1$ vectors of $P_N$ is a basis of $\Po$.
\end{proof}

By Theorem \ref{thm18:game_decomposition}, $P_N(P_N)^{\dag}$ is the orthogonal projection
onto $\Po$.
Next we give a direct definition for pure potential games, i.e., the following Theorem
\ref{thm10:game_decomposition}. This theorem also shows that pure potential games are exactly
the potential games that are normalized.

\begin{theorem}\label{thm10:game_decomposition}
	The pure potential games are exactly the games $(N,S,c)$
	in $\Gr_{[n;k_1,\dots,k_n]}$
	satisfying Eqns. \eqref{eqn:game_Potentialgames} and
	\eqref{eqn:game_Pureharmonicgame4}.
\end{theorem}

\begin{proof}
	The subspace consisting of games satisfying \eqref{eqn:game_Potentialgames} are $\Gr_P=\Po\oplus\Na$,
	and the subspace consisting of games satisfying \eqref{eqn:game_Pureharmonicgame4} are $\Po\oplus\Ha$.

	To prove this theorem, we only need to prove
	\begin{equation}
		(\Po\oplus\Na)\cap(\Po\oplus\Ha)=\Po.
		\label{}
	\end{equation}

	$(\Po\oplus\Na)\cap(\Po\oplus\Ha)\supset\Po$ holds naturally.
	Next we prove $(\Po\oplus\Na)\cap(\Po\oplus\Ha)\subset\Po$.

	Arbitrarily given $u\in(\Po\oplus\Na)\cap(\Po\oplus\Ha)$.
	Since $u\in\Po\oplus\Na$, $u=u_P+u_N$, where $u_P\in\Po$, $u_N\in\Na$.
	On the other hand, $u\in\Po\oplus\Ha=\Na^{\bot}$, then
	$0=B_NB_N^{\dag}u=B_NB_N^{\dag}u_P+B_NB_N^{\dag}u_N=0+u_N$. Hence
	$u=u_P$. That is, $(\Po\oplus\Na)\cap(\Po\oplus\Ha)\subset\Po$.
\end{proof}

\subsection{Subspaces of harmonic games}

\begin{theorem}\label{thm8:game_decomposition}
	The harmonic games are exactly the games $(N,S,c)$ in $\Gr_{[n;k_1,\dots,k_n]}$
	satisfying that
	\begin{equation}
		\forall s\in S,\quad
		\sum_{i=1}^n\left(\frac{1}{k_i}\sum_{x\in S^i}
		c_i(x,s^{-i})-
		c_i(s)\right)=0.
		\label{eqn:game_Harmonicgame}
	\end{equation}

	The harmonic subspace is
	\begin{equation}
		\Gr_H=\Na\oplus\Ha=\ker\left(
		\begin{bmatrix}
			I_k-\frac{1}{k_1}{\bf e}_1 & \cdots &
			I_k-\frac{1}{k_n}{\bf e}_n
		\end{bmatrix}
		\right),
		\label{eqn:game_Harmonicsubspace}
	\end{equation}
	and satisfies $\dim(\Gr_H)=(n-1)k+1$.
\end{theorem}

\begin{proof}
	Theorem \ref{thm18:game_decomposition} directly yields that \eqref{eqn:game_Harmonicsubspace} holds,
	and $\dim(\Gr_H)=(n-1)k+1$.

	Arbitrarily choose a game
	$u=(u_1,\dots,u_n)^T\in\Gr_{[n;k_1,\dots,k_n]}$, where each $u_i^T$ belongs to $\R^{k}$.
	For all $X=\ltimes_{j=1}^{n}X_j\in\Dt_k$, where each $X_j$ belongs to $\Dt_{k_j}$, we denote $X^{[1,i-1]}:=\ltimes
	_{j=1}^{i-1}X_j$, $X^{[i+1,n]}:=\ltimes_{j=i+1}^{n}X_j$.
	Then $u$ satisfies \eqref{eqn:game_Harmonicgame} iff, for all $X\in\Dt_k$,
	\begin{equation}
		\begin{split}
			0 &= \sum_{i=1}^n\left( \frac{1}{k_i}\sum_{z\in\Dt_{k_i}}u_iX^{[1,i-1]}zX^{[i+1,n]}
			-u_iX\right)\\
			&= \sum_{i=1}^n\left( \frac{1}{k_i}u_iX^{[1,i-1]}\left(\sum_{z\in\Dt_{k_i}}z\right)X^{[i+1,n]}
			-u_iX\right)\\
			&= \sum_{i=1}^n\left( \frac{1}{k_i}u_iX^{[1,i-1]}\left({\bf1}_{k_i}\right)X^{[i+1,n]}
			-u_iX\right)\\
			&= \sum_{i=1}^n\left( \frac{1}{k_i}u_iX^{[1,i-1]}\left({\bf1}_{k_i}{\bf1}^T_{k_i}X_i\right)X^{[i+1,n]}
			-u_iX\right)\\
			&= \sum_{i=1}^n\left( \frac{1}{k_i}u_i\left( I_{k^{[1,i-1]}}\otimes {\bf1}_{k_i\times k_i}
			\right) X -u_iX\right)\\
			&= \sum_{i=1}^n\left( \frac{1}{k_i}u_i\left( I_{k^{[1,i-1]}}\otimes {\bf1}_{k_i\times k_i}
			\otimes I_{k^{[i+1,n]}}\right) X -u_iX\right)\\
			&= \sum_{i=1}^{n}\left( \frac{1}{k_i}u_i{\bf e}_i-u_i \right)X
		\end{split}
		\label{}
	\end{equation}
	iff, $u\in\ker\left(\begin{bmatrix}
		I_k-\frac{1}{k_1}{\bf e}_1 & \cdots &
		I_k-\frac{1}{k_n}{\bf e}_n
		\end{bmatrix}\right)$.
\end{proof}

At the end of this subsection, we show that all harmonic games share a common mixed strategy Nash equilibrium ---
the uniformly mixed strategy profile.

\begin{theorem}\label{thm1:game_decomposition}
	Every harmonic game has the uniformly mixed strategy profile as its mixed strategy Nash equilibrium.
\end{theorem}

\begin{proof}
	Arbitrarily given a harmonic game $(N,S,c)$. Let the uniformly mixed strategy profile be $x^*=(x_1^*,\dots,
	x_n^*)\in\prod_{i=1}^{n}{\Dt S^i}$.

	From Theorem \ref{thm8:game_decomposition}, for all $s\in S$,
	\begin{equation}\begin{split}
		0 =& \sum_{i=1}^n\left(\frac{1}{k_i}\sum_{t\in S^i}
		c_i(t,s^{-i})-c_i(s)\right)\\
		=&  \frac{1}{k_j}\sum_{t\in S^j}
		c_j(t,s^{-j})-c_j(s)  + \\
		&\sum_{i=1,i\ne j}^n\left(\frac{1}{k_i}\sum_{t\in S^i}
		c_i(t,s^{-i})-c_i(s)\right).
	\end{split}
	\end{equation}

	Then for all players $j\in N$, all strategies $r\in S^j$,
	\begin{equation}
		\begin{split}
			0 =& \sum_{s^{-j}\in S^{-j}} \left(\frac{1}{k_j}\sum_{t\in S^j} c_j(t,s^{-j})-c_j(r,s^{-j}) \right)
			+\\& \sum_{s^{-j}\in S^{-j}}  \sum_{i=1,i\ne j}^n\left(\frac{1}{k_i}\sum_{t\in S^i}
			c_i(t,s^{-i})-c_i(s)\right)\\
			=& \frac{1}{k_j} \sum_{s^{-j}\in S^{-j}} \sum_{t\in S^j} c_j(t,s^{-j})- \sum_{s^{-j}\in S^{-j}}c_j(r,s^{-j}) +\\
			&\sum_{i=1,i\ne j}^n \left(\frac{1}{k_i} \sum_{s^{-j}\in S^{-j}} \sum_{t\in S^i}
			c_i(t,s^{-i})- \sum_{s^{-j}\in S^{-j}}c_i(s)\right)\\
			=& \frac{1}{k_j}\sum_{s\in S}c_j(s)-\sum_{s^{-j}\in S^{-j}}c_j(r,s^{-j})
			+\\& \sum_{i=1,i\ne j}^n \left(\frac{k_i}{k_i} \sum_{s^{-j}\in S^{-j}}
			c_i(s)- \sum_{s^{-j}\in S^{-j}}c_i(s)\right)\\
			=& \frac{1}{k_j}\sum_{s\in S}c_j(s)-\sum_{s^{-j}\in S^{-j}}c_j(r,s^{-j})\\
			=& \frac{k}{k_j} \left(\frac{1}{k}\sum_{s\in S}c_j(s)
			-\frac{1}{\frac{k}{k_j}}\sum_{s^{-j}\in S^{-j}}c_j(r,s^{-j})\right)\\
			=& \frac{k}{k_j} \left(c_j(x^*)-c_j(r,(x^*)^{-j})\right).
		\end{split}
	\end{equation}

	Hence $x^*$ is a mixed strategy Nash equilibrium of game $(N,S,c)$.
\end{proof}

From Theorem \ref{thm1:game_decomposition} and \cite[Theorem 5.5]{Candogan2011FlowsDecompositionofGames},
it follows that these two different harmonic games have similar properties.

\subsection{Decomposition of subspaces of finite games}

\begin{theorem}\label{thm3:game_decomposition}
	Consider the finite game space $\Gr_{[n;k_1,\dots,k_n]}$.
	The pure potential subspace $\Po$, the nonstrategic subspace
	$\Na$ and the pure harmonic subspace $\Ha$ satisfy that
	\begin{equation}
		\begin{split}
			\mathcal{P} &= \im(B_PB_P^{\dag}-B_NB_N^{\dag}),\\
			\mathcal{N} &= \im(B_NB_N^{\dag}),\\
			\mathcal{H} &= \im(I_{nk}-B_PB_P^{\dag}).
		\end{split}
		\label{eqn:game_subspace}
	\end{equation}
	$B_PB_P^{\dag}-B_NB_N^{\dag},B_NB_N^{\dag},I_{nk}-B_PB_P^{\dag}$
	are orthogonal projections onto these three subspaces.
	Under the conventional inner product,
	$\Po\bot\Na$, $\Po\bot\Ha$, $\Na\bot\Ha$.

	Every game $u\in\Gr_{[n;k_1,\dots,k_n]}$ has a unique direct-sum
	decomposition
	$u=u_{\mathcal{P}}+u_{\mathcal{N}}+u_{\mathcal{H}}$,
	where $u_{\mathcal{P}}\in\Po,u_{\mathcal{N}}\in\Na,u_{\mathcal{H}}\in\Ha$.
	Actually,
	$u_{\mathcal{P}}=(B_PB_P^{\dag}-B_NB_N^{\dag})u\in\mathcal{P}$,
	$u_{\mathcal{N}}=B_NB_N^{\dag}u\in\mathcal{N}$,
	$u_{\mathcal{H}}=(I_{nk}-B_PB_P^{\dag})u\in\mathcal{H}$.

\end{theorem}

\begin{proof}
	By Theorems  \ref{thm12:game_decomposition},
	\ref{prop_game_Purepotentialsubspace} and Proposition
	\ref{prop_game_Pureharmonicsubspace},
	$\Na=\im(B_NB_N^{\dag})$, $\Ha=\im(I_{nk}-B_PB_P^{\dag})$.

	By the definition of Moore-Penrose inverses, $B_NB_N^{\dag}$ and
	$I_{nk}-B_PB_P^{\dag}$ are both orthogonal projections.
	Since $\im(B_P)\supset\im(B_N)$, from Proposition
	\ref{prop_MPinverse_rangeinclusion}, $B_PB_P^{\dag}B_NB_N^{\dag}=
	B_NB_N^{\dag}B_PB_P^{\dag}=B_NB_N^{\dag}$.
	Then $(B_PB_P^{\dag}-B_NB_N^{\dag})^2=(B_PB_P^{\dag}-B_NB_N^{\dag})$, i.e.,
	$B_PB_P^{\dag}-B_NB_N^{\dag}$ is also an orthogonal projection.
	One also has $B_NB_N^{\dag}(I_{nk}-B_PB_P^{\dag})=
	(I_{nk}-B_PB_P^{\dag})B_NB_N^{\dag}=0_{(nk)\times(nk)}$,
	$B_NB_N^{\dag}(
	B_PB_P^{\dag}-B_NB_N^{\dag})=(B_PB_P^{\dag}-B_NB_N^{\dag})
	B_NB_N^{\dag}=(I_{nk}-B_PB_P^{\dag})(
	B_PB_P^{\dag}-B_NB_N^{\dag})=(B_PB_P^{\dag}-B_NB_N^{\dag})
	(I_{nk}-B_PB_P^{\dag})
	=0_{(nk)\times(nk)}$. Hence $\im(B_PB_P^{\dag}-B_NB_N^{\dag})=
	\Po$, and $\Po\bot\Na$, $\Po\bot\Ha$, $\Na\bot\Ha$.

	Arbitrarily given $u\in\Gr_{[n;k_1,\dots,k_n]}$, it is obviously that
	$u=(B_PB_P^{\dag}-B_NB_N^{\dag})u+
	B_NB_N^{\dag}u+(I_{nk}-B_PB_P^{\dag})u$ is such a decomposition.
	Next we prove that the decomposition is unique.
	There are $u_{\Po}\in
	\Po$, $u_{\Na}\in\Na$ and $u_{\Ha}\in\Ha$ such that
	$u=u_{\Po}+u_{\Na}+u_{\Ha}$. From Eqn.
	\eqref{eqn:game_subspace},
	\begin{equation}
		\begin{split}
			&(B_PB_P^{\dag}-B_NB_N^{\dag})u \\
			=& (B_PB_P^{\dag}-B_NB_N^{\dag})u_{\Po} +
			(B_PB_P^{\dag}-B_NB_N^{\dag})u_{\Na} + \\&
			(B_PB_P^{\dag}-B_NB_N^{\dag})u_{\Ha}\\
			=& u_{\Po}+0+0=u_{\Po},\\
			&B_NB_N^{\dag}u \\
			=& B_NB_N^{\dag}u_{\Po} +
			B_NB_N^{\dag}u_{\Na} +
			B_NB_N^{\dag}u_{\Ha}\\
			=& 0+u_{\Na}+0=u_{\Na},\\
			&(I_{nk}-B_PB_P^{\dag})u \\
			=& (I_{nk}-B_PB_P^{\dag})u_{\Po} +
			(I_{nk}-B_PB_P^{\dag})u_{\Na} + \\&
			(I_{nk}-B_PB_P^{\dag})u_{\Ha}\\
			=& 0+0+u_{\Ha}=u_{\Ha}.
		\end{split}
		\label{}
	\end{equation}
	That is, the decomposition is unique.
\end{proof}

From Proposition \ref{prop_orthogonalprojection1}, Theorems \ref{thm18:game_decomposition} and
\ref{thm3:game_decomposition}, the following Proposition \ref{thm15:game_decomposition} follows.

\begin{proposition}\label{thm15:game_decomposition}
	\begin{equation}
		\begin{split}
			B_PB_P^{\dag} &= B_NB_N^{\dag}+P_NP_N^{\dag}.
		\end{split}
		\label{}
	\end{equation}
\end{proposition}

\section{Polynomial representation for orthogonal projections onto subspaces of finite games}
\label{sec:explicitProjontoGameSubsapce}

In Section \ref{sec:ProjontoGameSubsapce},
polynomial representation for the orthogonal projection onto
the nonstrategic subspace $\Na$, i.e.,
$\left(\frac{1}{k_1}{\bf e}_1\oplus\cdots\oplus\frac{1}{k_n}{\bf e}_n\right)$,
is given in Theorem \ref{thm12:game_decomposition}.
In the sequel, based on the results in Section \ref{sec:ProjontoGameSubsapce},
we give polynomial representation for orthogonal projections onto
all the other subspaces of finite games and for potential functions of potential games.
We start with the pure potential subspace.

\subsection{Algorithms for calculating the orthogonal projections onto the pure potential subspace}

By Propositions \ref{prop_linearspace}, \ref{prop_groupMPinverse}, and
Theorem \ref{thm18:game_decomposition}, the following theorem holds.

\begin{theorem}\label{thm9:game_decomposition}
	\begin{equation}
		\begin{split}
		&\Po =\im\left(
		\begin{bmatrix}
			I_k-\frac{1}{k_1}{\bf e}_1 \\ \vdots \\
			I_k-\frac{1}{k_n}{\bf e}_n
		\end{bmatrix}
		\right) =\\& \im\left(
		\begin{bmatrix}
			I_k-\frac{1}{k_1}{\bf e}_1 \\ \vdots \\
			I_k-\frac{1}{k_n}{\bf e}_n
		\end{bmatrix}
		\begin{bmatrix}
			I_k-\frac{1}{k_1}{\bf e}_1 \\ \vdots \\
			I_k-\frac{1}{k_n}{\bf e}_n
		\end{bmatrix}^{\dag}
		\right)=\\
		& \im\left(
		\begin{bmatrix}
			I_k-\frac{1}{k_1}{\bf e}_1 \\ \vdots \\
			I_k-\frac{1}{k_n}{\bf e}_n
		\end{bmatrix}\left(
		\sum_{i=1}^{n}\left(I_k-\frac{1}{k_i}{\bf e}_i\right)
		\right)^{\sharp} 
		\begin{bmatrix}
			I_k-\frac{1}{k_1}{\bf e}_1 \\ \vdots \\
			I_k-\frac{1}{k_n}{\bf e}_n
		\end{bmatrix}^T\right),
		\end{split}
		\label{eqn:game_Purepotentialgame}
	\end{equation}
	where ${\bf e}_1,\dots,{\bf e}_n$ are defined in \eqref{eqn:game_E_i}.
\end{theorem}

If we obtain the representation for $\left(\sum_{i=1}^{n}\left(I_k-
\frac{1}{k_i}{\bf e}_i\right)\right)^{\sharp}$, we will obtain the polynomial representation
for the orthogonal projection onto $\Po$.

First we design an algorithm to calculate\\ $\left(\sum_{i=1}^{n}\left(I_k-
\frac{1}{k_i}{\bf e}_i\right)\right)^{\sharp}$.
The following proposition plays an important role in designing this algorithm.

\begin{proposition}\label{prop_game_eNs}
	Consider the finite game space $\Gr_{[n;k_1,\dots,k_n]}$.
	${\bf e}_{N_s}$, $N_s\subset[1,n]$ (defined in \eqref{eqn:game_eNs}) are linearly independent.
\end{proposition}

\begin{proof}
	Let $c_{N_s}\in\R$, $N_s\subset [1,n]$, and $\sum_{{N_s\subset[1,n]}}
	{c}_{N_s}{\bf e}_{N_s}=0$.

	First we consider the $(1,k)$-entry of ${\bf e}_{N_s}$, $N_s\subset
	[1,n]$.
	It can be seen that ${\bf e}_{[1,n]}(1,k)=1$, and for all $N_s\subsetneq
	[1,n]$, ${\bf e}_{N_s}(1,k)=0$. Hence $c_{[1,n]}=0$. Remove
	${\bf e}_{[1,n]}$.

	Second we consider the $(1,\frac{k}{k_1})$-entry of ${\bf e}_{N_s}$,
	$N_s\subsetneq[1,n]$. It can be seen that ${\bf e}_{[2,n]}(1,\frac
	{k}{k_i})=1$, and for all $[2,n]\ne N_s\subsetneq[1,n]$,
	${\bf e}_{N_s}(1,\frac{k}{k_i})=0$. Hence $c_{[2,n]}=0$.
	Remove ${\bf e}_{[2,n]}$.

	Similarly we have for all $i\in[1,n]$, $c_{[1,n]\setminus\{i\}}=0$.
	Remove ${\bf e}_{N_s}$, $|N_s|=n-1$.
	
	Similarly for all $N_s\subset[1,n]$ satisfying $|N_s|=n-2$,
	$c_{N_s}=0$. Remove ${\bf e}_{N_s}$, $|N_s|=n-2$.

	Repeat this procedure again and again, and finally we have
	for all $N_s\subset[1,n]$, $c_{N_s}=0$.

	Based on the above analysis, ${\bf e}_{N_s}$, $N_s\subset[1,n]$
	are linearly independent.
\end{proof}

By Propositions \ref{prop_groupinverse} and  \ref{prop_game_eNs},
$\left(\sum_{i=1}^{n}\left(I_k-\frac{1}{k_i}{\bf e}_i\right)\right)^{\sharp}$
is in the form of $\sum_{N_s\subset[1,n]}c_{N_s}{\bf e}_{N_s}$, where
all coefficients $c_{N_s}$ belong to $\R$ and are unique.
Let $X=\sum_{N_s\subset[1,n]}d_{N_s}{\bf e}_{N_s}$, where $d_{N_s}\in\R$
are to be determined. Then
$\left(\sum_{i=1}^{n}\left(I_k-\frac{1}{k_i}{\bf e}_i\right)\right)^2X$
is in the form of $\sum_{N_s\subset[1,n]}e_{N_s}{\bf e}_{N_s}$,
where $e_{N_s}$ are polynomials of $d_{N_s}$, $N_s\subset[1,n]$.
From Proposition \ref{prop_groupinverse_Oredomain}, it follows that
linear equation
\begin{equation}\label{eqn:game_Pureharmonicgame5}
	\begin{split}
	&\left(\sum_{i=1}^{n}\left(I_k-\frac{1}{k_i}{\bf e}_i\right)\right)^2
	\left(\sum_{N_s\subset[1,n]}d_{N_s}{\bf e}_{N_s}\right)=\\
	&\sum_{i=1}^{n}\left(I_k-\frac{1}{k_i}{\bf e}_i\right)
	\end{split}
\end{equation}
has solution.

\begin{algo}\label{algo3_game_Pureharmonicgamesubspace}
	Find a solution $\{d^0_{N_s}|N_s\subset[1,n]\}$
	of Eqn. \eqref{eqn:game_Pureharmonicgame5}.
	From Proposition \ref{prop_groupinverse_Oredomain},
	$\left(\sum_{i=1}^{n}\left(I_k-\frac{1}{k_i}{\bf e}_i\right)\right)
	^{\sharp}=
	\left(\sum_{i=1}^{n}\left(I_k-\frac{1}{k_i}{\bf e}_i\right)\right)
	\left(\sum_{N_s\subset[1,n]}d^0_{N_s}{\bf e}_{N_s}\right)^2:=\\
	\left(\sum_{N_s\subset[1,n]}e^0_{N_s}{\bf e}_{N_s}\right)$.
\end{algo}


By using Algorithm \ref{algo3_game_Pureharmonicgamesubspace}, for any given $n$,
one can obtain the polynomial representation for
$\left(\sum_{i=1}^{n}\left(I_k-\frac{1}{k_i}{\bf e}_i\right)
\right)^{\sharp}$.

\begin{example}\label{exam1:game_subspace}
	For $n=2$: By using Algorithm \ref{algo3_game_Pureharmonicgamesubspace}, one has
	\begin{equation}
		\left(\sum_{i=1}^{2}\left(I_k-\frac{1}{k_i}{\bf e}_i\right)\right)^{\sharp}=
		\frac{1}{2}I_k+\frac{1}{2}\frac{{\bf e}_1}{k_1}+\frac{1}{2}\frac{{\bf e}_2}{k_2}
		-\frac{3}{2}\frac{{\bf e}_1{\bf e}_2}{k_1k_2}.
		\label{}
	\end{equation}
\end{example}

\begin{example}\label{exam2:game_subspace}
	For $n=3$: By using Algorithm \ref{algo3_game_Pureharmonicgamesubspace}, one has
	\begin{equation}\begin{split}
		&\left(\sum_{i=1}^{3}\left(I_k-\frac{1}{k_i}{\bf e}_i\right)\right)^{\sharp}= \\&
		\frac{1}{3}I_k+\frac{1}{6}\frac{{\bf e}_1}{k_1}+\frac{1}{6}\frac{{\bf e}_2}{k_2}
		+\frac{1}{6}\frac{{\bf e}_3}{k_3}+ \\&
		\frac{1}{3}\frac{{\bf e}_1{\bf e}_2}{k_1k_2}+\frac{1}{3}\frac{{\bf e}_1{\bf e}_3}{k_1k_3}
		+\frac{1}{3}\frac{ {\bf e}_2{\bf e}_3}{k_2k_3}
		-\frac{11}{6}\frac{{\bf e}_1{\bf e}_2{\bf e}_3}{k_1k_2k_3}.
	\end{split}
	\end{equation}
\end{example}

Second by observing the above specific examples,
we obtain the polynomial representation for\\
$\left(\sum_{i=1}^{n}\left(I_k-\frac{1}{k_i}{\bf e}_i\right)
\right)^{\sharp}$ as a function of $n$ and a linear combination of
${\bf e}_{N_s}$, $N_s\subset[1,n]$. That is, one can directly use Theorem \ref{thm14:game_decomposition}
to calculate $\left(\sum_{i=1}^{n}\left(I_k-\frac{1}{k_i}{\bf e}_i\right)\right)^{\sharp}$.
Note that Theorem \ref{thm14:game_decomposition} is a closed form, i.e.,
no matter whether $n$ is given, one knows the polynomial form of
$\left(\sum_{i=1}^{n}\left(I_k-\frac{1}{k_i}{\bf e}_i\right)\right)^{\sharp}$.
Such results cannot be found
in \cite{Cheng2014PotentialGame} or \cite{Candogan2011FlowsDecompositionofGames}.

\begin{theorem}\label{thm14:game_decomposition}
	Consider the finite game space $\Gr_{[n;k_1,\dots,k_n]}$.
	\begin{equation}
		\begin{split}
			&\left(\sum_{i=1}^{n}\left(I_k-\frac{1}{k_i}{\bf e}_i\right)
			\right)^{\sharp} = \\&
			\sum_{j=0}^{n-1}\sum_{1\le i_1<\cdots<i_j\le n}
			\frac{1}{(n-j)C_n^j}\prod_{l=1}^{j}
			\frac{{\bf e}_{i_l}}
			{k_{i_l}}-
			\left(\sum_{i=1}^{n}\frac{1}{i}\right)
			\prod_{i=1}^{n}\frac{{\bf e}_{i}}
			{k_{i}},
		\end{split}
		\label{eqn:game_PurepotentialgameGroupinverse}
	\end{equation}
	where $\sum_{1\le i_1<\cdots< i_0\le n}\frac{1}{
	(n-0)C_n^0}\prod_{l=1}^{0}\frac{{\bf e}_{i_l}}{k_{i_l}}=\frac{1}{n}I_k$.
\end{theorem}

\begin{proof}
	Let $X$ denote the right hand side of Eqn. \eqref{eqn:game_PurepotentialgameGroupinverse}. Since
	$\sum_{i=1}^{n}\left(I_k-\frac{1}{k_i}{\bf e}_i\right)$
	and $X$ are both polynomials of ${\bf e}_i$, $i\in[1,n]$,
	$\left(\sum_{i=1}^{n}\left(I_k-\frac{1}{k_i}{\bf e}_i\right)\right)
	X=X\left(\sum_{i=1}^{n}\left(I_k-\frac{1}{k_i}{\bf e}_i\right)\right)$.

	We claim that
	\begin{equation}
		\left(\sum_{i=1}^{n}\left(I_k-\frac{1}{k_i}{\bf e}_i\right)
		\right)X=I_k-\prod_{i=1}^{n}\frac{{\bf e}_{i}}
		{k_{i}}.
		\label{eqn:game_PurepotentialgameGroupinverse1}
	\end{equation}

	It can be verified that the coefficient of $I_k$ of the left
	hand side of Eqn. \eqref{eqn:game_PurepotentialgameGroupinverse1}
	is $n\cdot\frac{1}{(n-0)C_n^0}=1$.

	For all $1\le j<n$, and all $1\le i_1<\cdots <i_j\le n$,
	the terms $\prod_{l=1}^{j}\frac{ {\bf e}_{k_l}}{k_{i_l}}$ only come
	from $(nI_k)\left(\frac{1}{(n-j)C_n^j}\prod_{l=1}^{j}\frac{{\bf e}_{i_l}}
	{k_{i_l}}\right)$, $\left(-\sum_{i=1}^n\left(\frac{1}{k_i}{\bf e}_i\right)\right)
	\left(\frac{1}{(n-j)C_n^j}\prod_{l=1}^{j}\frac{{\bf e}_{i_l}}
	{k_{i_l}}\right)$
	and $\left(-\sum_{i=1}^n\left(\frac{1}{k_i}{\bf e}_i\right)\right)\\
	\left(\sum_{N_s\subset\{i_1,\dots,i_j\},|N_s|=j-1}\frac{1}{(n-(j-1))C_n^{j-1}}
	\frac{ {\bf e}_{N_s}}{\prod_{l\in N_s}k_l}\right)$,
	and their coefficients are $\frac{n}{ (n-j)C_n^j}$, $-\frac{j}
	{(n-j)C_n^j}$ and\\ $-\frac{C_j^{j-1}}{(n-(j-1))C_n^{j-1}}$,
	respectively. Then the coefficient of $\prod_{l=1}^{j}\frac{ {\bf e}_{k_l}}{k_{i_l}}$
	of the left hand side of Eqn.
	\eqref{eqn:game_PurepotentialgameGroupinverse1} is equal to
	$\frac{n}{ (n-j)C_n^j}-\frac{j}
	{(n-j)C_n^j}-\frac{C_j^{j-1}}{(n-(j-1))C_n^{j-1}}=0$.

	The term $\prod_{i=1}^{n}\frac{ {\bf e}_i}{k_i}$ comes from
	$\left(\sum_{i=1}^{n}\left(I_k-\frac{1}{k_i}{\bf e}_i\right)
	\right)\\
	\left(\sum_{i=1}^{n}\frac{1}{i}\right)
	\prod_{i=1}^{n}\frac{{\bf e}_{i}}{k_{i}}$
	and $\left(-\sum_{i=1}^n\left(\frac{1}{k_i}{\bf e}_i\right)\right)\\
	\left(\sum_{N_s\subset[1,n],|N_s|=n-1}\frac{1}{(n-(n-1))C_n^{n-1}}
	\frac{ {\bf e}_{N_s}}{\prod_{l\in N_s}k_l}\right)$, and their
	coefficients are $n\sum_{i=1}^{n}\frac{1}{i}-n\sum_{i=1}^{n}\frac{1}{i}=0$
	and\\ $-\frac{C_{n}^{n-1}}{(n-(n-1))C_n^{n-1}}=-1$. Hence
	the coefficient of $\prod_{i=1}^{n}\frac{ {\bf e}_{i}}{k_{i}}$
	of the left hand side of Eqn.
	\eqref{eqn:game_PurepotentialgameGroupinverse1} is equal to
	$0-1=-1$.

	Based on these analysis,
	Eqn. \eqref{eqn:game_PurepotentialgameGroupinverse1} holds.

	We also have
	\begin{equation}
		\begin{split}
			&\left(\sum_{i=1}^{n}\left(I_k-\frac{1}{k_i}{\bf e}_i\right)
			\right)\left(I_k-\prod_{i=1}^{n}\frac{{\bf e}_{i}}{k_{i}}\right) \\
			=&nI_k-n\prod_{i=1}^{n}\frac{ {\bf e}_i}{k_i}-
			\sum_{i=1}^{n}\frac{ {\bf e}_i}{k_i}+
			n\prod_{i=1}^{n}\frac{ {\bf e}_i}{k_i} \\
			=&nI_k-\sum_{i=1}^{n}\frac{ {\bf e}_i}{k_i},
		\end{split}
		\label{eqn:game_PurepotentialgameGroupinverse2}
	\end{equation}and
	\begin{equation}
		\begin{split}
			&\left(I_k-\prod_{i=1}^{n}\frac{{\bf e}_{i}}{k_{i}}\right)X = \\&
			X-\left(\sum_{j=0}^{n-1}\frac{C_n^j}{ (n-j)C_n^j}-
			\sum_{i=1}^n\frac{1}{i}\right)
			\prod_{i=1}^{n}\frac{{\bf e}_{i}}{k_{i}}=X.
		\end{split}
		\label{eqn:game_PurepotentialgameGroupinverse3}
	\end{equation}

	From \eqref{eqn:game_PurepotentialgameGroupinverse1},
	\eqref{eqn:game_PurepotentialgameGroupinverse2},
	\eqref{eqn:game_PurepotentialgameGroupinverse3}, and
	$\left(\sum_{i=1}^{n}\left(I_k-\frac{1}{k_i}{\bf e}_i\right)\right)
	X=X\left(\sum_{i=1}^{n}\left(I_k-\frac{1}{k_i}{\bf e}_i\right)\right)$,
	\begin{equation}
		X=\left(\sum_{i=1}^{n}\left(I_k-\frac{1}{k_i}{\bf e}_i\right)
		\right)^{\sharp}.
		\label{}
	\end{equation}
\end{proof}

\begin{remark}
	From Proposition \ref{prop_game_eNs}, ${\bf e}_{N_s}$, $N_s\subset[1,n]$ are linearly
	independent, so the representation for\\ $\left(\sum_{i=1}^{n}\left(I_k-\frac{1}{k_i}{\bf e}_i\right)
	\right)^{\sharp}$ shown in \eqref{eqn:game_PurepotentialgameGroupinverse} cannot be simpler.
\end{remark}

\subsection{Polynomial representation for orthogonal projections onto subspaces of finite games}

By Theorems \ref{thm9:game_decomposition} and \ref{thm14:game_decomposition},
polynomial representation for orthogonal
projections onto the subspaces of pure potential games, nonstrategic games
and pure harmonic games shown in Eqn. \eqref{eqn:game_subspace} are obtained as below.

\begin{theorem}\label{thm17:game_decomposition}
	Consider the finite game space $\Gr_{[n;k_1,\dots,k_n]}$.
	The polynomial representation for orthogonal projections onto the pure potential subspace $\Po$,
	the nonstrategic subspace $\Na$, the pure harmonic subspace $\Ha$, the potential subspace
	$\Gr_P$, the harmonic subspace $\Gr_H$ are as below: 
	\begin{equation}
		\begin{split}
			& P_NXP_N^T,	\quad
			\left(\frac{1}{k_1}{\bf e}_1\oplus\cdots\oplus
			\frac{1}{k_n}{\bf e}_n\right), \\
			& I_{nk}-P_NXP_N^T -\left(\frac{1}{k_1}{\bf e}_1\oplus\cdots\oplus
			\frac{1}{k_n}{\bf e}_n\right),\\
			& P_NXP_N^T + \left(\frac{1}{k_1}{\bf e}_1\oplus\cdots\oplus
			\frac{1}{k_n}{\bf e}_n\right),\\
			& I_{nk}-P_NXP_N^T,
		\end{split}
		\label{eqn:explicitrepresentationforOPOontogamesubspace}
	\end{equation}
	where $P_N$ is defined in \eqref{eqn:PN}, $X$ is the right hand side of
	\eqref{eqn:game_PurepotentialgameGroupinverse}, and ${\bf e}_i$'s are defined in
	\eqref{eqn:game_E_i}.
\end{theorem}

\subsection{Polynomial representation for potential functions}

In this subsection, we give polynomial representation for potential functions.

\begin{theorem}\label{thm4:game_decomposition}
	Consider the finite game space $\Gr_{[n;k_1,\dots,k_n]}$. If a game $G\in \Gr_{[n;k_1,\dots,k_n]}$
	is potential, then the structure matrix of its potential functions are the first $k$ elements of
	\begin{equation}
		\begin{bmatrix}
			I_k &  & & \\
			-\frac{1}{k_1}E_1^T & I_{\frac{k}{k_1}} & & \\
			\vdots & & \ddots & \\
			-\frac{1}{k_n}E_n^T &&& I_{\frac{k}{k_n}}
		\end{bmatrix}
		\begin{bmatrix}
			P_N^{\dag} \\ B_N^{\dag}
		\end{bmatrix} G
		+ c{\bf1}_{k+\sum_{i=1}^{n}{\frac{k}{k_i}}},
		\label{eqn:potentialfunction}
	\end{equation}
	where $c\in\R$ is arbitrary, $P_N^{\dag}=XP_N^T$, $P_N$ is defined in \eqref{eqn:PN},
	$X$ is the right hand side of
	\eqref{eqn:game_PurepotentialgameGroupinverse}, $B_N^{\dag}=(\frac{1}{k_1}E_1^T\oplus\cdots\oplus
	\frac{1}{k_n}E_n^T)$.
\end{theorem}

\begin{proof}
	In view of \eqref{eqn:potentialequation4}, in order to prove this theorem, we only need to prove
	$$B_P\begin{bmatrix}
		I_k &  & & \\
		-\frac{1}{k_1}E_1^T & I_{\frac{k}{k_1}} & & \\
		\vdots & & \ddots & \\
		-\frac{1}{k_n}E_n^T &&& I_{\frac{k}{k_n}}
	\end{bmatrix}
	\begin{bmatrix}
		P_N^{\dag} \\ B_N^{\dag}
	\end{bmatrix} G=G.$$

	By \eqref{eqn:game_Purepotential3} and
	Theorem \ref{thm15:game_decomposition},
	\begin{equation}
		\begin{split}
			&B_P\begin{bmatrix}
				I_k &  & & \\
				-\frac{1}{k_1}E_1^T & I_{\frac{k}{k_1}} & & \\
				\vdots & & \ddots & \\
				-\frac{1}{k_n}E_n^T &&& I_{\frac{k}{k_n}}
			\end{bmatrix}
			\begin{bmatrix}
				P_N^{\dag} \\ B_N^{\dag}
			\end{bmatrix}G \\
			=&\begin{bmatrix}
				P_N & B_N
			\end{bmatrix}\begin{bmatrix}
				P_N^{\dag} \\ B_N^{\dag}
			\end{bmatrix}G
			= (P_NP_N^{\dag}+B_NB_N^{\dag})G \\
			=& B_PB_P^{\dag}G=G.
		\end{split}
		\label{}
	\end{equation}

\end{proof}

By Theorems \ref{thm17:game_decomposition} and \ref{thm4:game_decomposition},
in order to testify whether a finite game belongs to one of the five subspaces described in
Theorem \ref{thm17:game_decomposition}, one only needs to use the number $n$ of players and those
$k_1,\dots,k_n$
of their strategies to calculate the orthogonal projection $P_\pi$ onto them, where $P_{\pi}$
denotes the orthogonal projection onto one of the five subspaces.
Then game $u\in\Gr_{[n;k_1,\dots,k_n]}$ belongs to the corresponding subspace iff $P_{\pi}u=u$.
Furthermore, if $u$ is potential, one can use the orthogonal projection $P_{\Gr_P}$ onto the potential subspace
to calculate its potential functions according to Theorem \ref{thm4:game_decomposition}.


Besides, although
closed forms of the orthogonal projections onto the subspaces of $2$-player finite games  are given in
\cite[Subsection 4.3]{Candogan2011FlowsDecompositionofGames}, closed forms for the general case are not shown.
In order to obtain the closed forms of the orthogonal projections onto the subspaces of
finite games shown in \cite{Candogan2011FlowsDecompositionofGames}, one needs to find
the closed forms of $\dt_0^{\dag}$ and $D^\dag$ in
\cite[Theorem 4.1]{Candogan2011FlowsDecompositionofGames}, where $\dt_0=\sum_{i=1}^
{M}D_i$ and $D=[D_1^*,...,D_M^*]^*$, $(\cdot)^*$ denotes the adjoint opertor of $\cdot$,
$D_1,...,D_M$ are corresponding linear operators. The closed form of $D^\dag$ is given
in \cite[Lemma 4.4]{Candogan2011FlowsDecompositionofGames}, while the closed form
of $\dt_0^\dag$ is not shown.

Next we give some illustrative examples to show the advantage of our results.

\subsection{Illustrative examples}

\begin{example}\cite{Cheng2014PotentialGame}\label{exam1:orthogonal_operator_game}
	Consider an arbitrary symmetric game with $n=3$ and $k_1=k_2=k_3=2$. It is proved in \cite{Cheng2014PotentialGame}
	that such symmetric games are always potential via proving that the corresponding potential equations always
	have solution. Next we show how to use Theorem \ref{thm17:game_decomposition} to verify this conclusion.
	The process of \cite{Cheng2014PotentialGame} is constructive. However our following process is
	mechanical, and can be finished by a computer program.
	
    First we obtain the orthogonal projection onto the subspace of potential games $\Gr_P$ according to
	Theorem \ref{thm17:game_decomposition} as in Eqn. \eqref{eqn_projection_exam1}.
	\begin{figure*}
		\begin{equation}\label{eqn_projection_exam1}
	P_{\Gr_P}=\frac{1}{48}
	\begin{bmatrix}
	38,  4,  4,  2, 10, -4, -4, -2,  5,  1, -5, -1, -5, -1,  5,  1,  5, -5,  1, -1, -5,  5, -1,  1\\
	4, 38,  2,  4, -4, 10, -2, -4,  1,  5, -1, -5, -1, -5,  1,  5, -5,  5, -1,  1,  5, -5,  1, -1\\
	4,  2, 38,  4, -4, -2, 10, -4, -5, -1,  5,  1,  5,  1, -5, -1,  1, -1,  5, -5, -1,  1, -5,  5\\
	2,  4,  4, 38, -2, -4, -4, 10, -1, -5,  1,  5,  1,  5, -1, -5, -1,  1, -5,  5,  1, -1,  5, -5\\
	10, -4, -4, -2, 38,  4,  4,  2, -5, -1,  5,  1,  5,  1, -5, -1, -5,  5, -1,  1,  5, -5,  1, -1\\
	-4, 10, -2, -4,  4, 38,  2,  4, -1, -5,  1,  5,  1,  5, -1, -5,  5, -5,  1, -1, -5,  5, -1,  1\\
	-4, -2, 10, -4,  4,  2, 38,  4,  5,  1, -5, -1, -5, -1,  5,  1, -1,  1, -5,  5,  1, -1,  5, -5\\
	 -2, -4, -4, 10,  2,  4,  4, 38,  1,  5, -1, -5, -1, -5,  1,  5,  1, -1,  5, -5, -1,  1, -5,  5\\
	  5,  1, -5, -1, -5, -1,  5,  1, 38,  4, 10, -4,  4,  2, -4, -2,  5, -5, -5,  5,  1, -1, -1,  1\\
	  1,  5, -1, -5, -1, -5,  1,  5,  4, 38, -4, 10,  2,  4, -2, -4, -5,  5,  5, -5, -1,  1,  1, -1\\
	 -5, -1,  5,  1,  5,  1, -5, -1, 10, -4, 38,  4, -4, -2,  4,  2, -5,  5,  5, -5, -1,  1,  1, -1\\
	 -1, -5,  1,  5,  1,  5, -1, -5, -4, 10,  4, 38, -2, -4,  2,  4,  5, -5, -5,  5,  1, -1, -1,  1\\
	 -5, -1,  5,  1,  5,  1, -5, -1,  4,  2, -4, -2, 38,  4, 10, -4,  1, -1, -1,  1,  5, -5, -5,  5\\
	-1, -5,  1,  5,  1,  5, -1, -5,  2,  4, -2, -4,  4, 38, -4, 10, -1,  1,  1, -1, -5,  5,  5, -5\\
	 5,  1, -5, -1, -5, -1,  5,  1, -4, -2,  4,  2, 10, -4, 38,  4, -1,  1,  1, -1, -5,  5,  5, -5\\
	 1,  5, -1, -5, -1, -5,  1,  5, -2, -4,  2,  4, -4, 10,  4, 38,  1, -1, -1,  1,  5, -5, -5,  5\\
	 5, -5,  1, -1, -5,  5, -1,  1,  5, -5, -5,  5,  1, -1, -1,  1, 38, 10,  4, -4,  4, -4,  2, -2\\
	-5,  5, -1,  1,  5, -5,  1, -1, -5,  5,  5, -5, -1,  1,  1, -1, 10, 38, -4,  4, -4,  4, -2,  2\\
	 1, -1,  5, -5, -1,  1, -5,  5, -5,  5,  5, -5, -1,  1,  1, -1,  4, -4, 38, 10,  2, -2,  4, -4\\
	-1,  1, -5,  5,  1, -1,  5, -5,  5, -5, -5,  5,  1, -1, -1,  1, -4,  4, 10, 38, -2,  2, -4,  4\\
	-5,  5, -1,  1,  5, -5,  1, -1,  1, -1, -1,  1,  5, -5, -5,  5,  4, -4,  2, -2, 38, 10,  4, -4\\
	 5, -5,  1, -1, -5,  5, -1,  1, -1,  1,  1, -1, -5,  5,  5, -5, -4,  4, -2,  2, 10, 38, -4,  4\\
	-1,  1, -5,  5,  1, -1,  5, -5, -1,  1,  1, -1, -5,  5,  5, -5,  2, -2,  4, -4,  4, -4, 38, 10\\
	 1, -1,  5, -5, -1,  1, -5,  5,  1, -1, -1,  1,  5, -5, -5,  5, -2,  2, -4,  4, -4,  4, 10, 38\\
	\end{bmatrix}
	\end{equation}
	\end{figure*}

	Second, we verify whether such symmetric games $G=[a, b, b, d, c, e, e, f, a, b, c, e, b, d, e, f, a, c, b, e, b, e, d, f]$
	satisfy that $P_{\Gr_P}G^T=G^T$. It is directly obtained that $P_{\Gr_P}G^T=G^T$ holds no matter what real numbers
	$a,b,c,d,e,f$ are. Hence such symmetric games are always potential.

    Third, we obtain their potential
    functions according to Theorem \ref{thm4:game_decomposition} as $P(x)=[ (7a)/8 + b/2 - (7c)/8 + d/8 - e/2 - f/8,
         b/2 - a/8 + c/8 + d/8 - e/2 - f/8,
         b/2 - a/8 + c/8 + d/8 - e/2 - f/8,
         c/8 - b/2 - a/8 + d/8 + e/2 - f/8,
         b/2 - a/8 + c/8 + d/8 - e/2 - f/8,
         c/8 - b/2 - a/8 + d/8 + e/2 - f/8,
         c/8 - b/2 - a/8 + d/8 + e/2 - f/8, c/8 - b/2 - a/8 - (7d)/8 + e/2 + (7f)/8]x+c_0$, where
         $x=\ltimes_{i=1}^{3}{x_i}\in\Dt_8$, $c_0\in\R$ is arbitrary.
	If we choose $a = 1, b = 1, c = 2, d = -1, e = 1, f =-1$, and
    $c_0=-9/8$, we obtain the following
    potential function  $P(x)=[ -2,-1,-1,-1,-1,-1,-1,-1]x$, which is the same as the one obtained in
    \cite{Cheng2014PotentialGame}.
\end{example}


Now it is known that every symmetric game with three players and two strategies is potential. Then
is every finite symmetric game potential? Next we give a negative answer.

\begin{example}\label{exam2:orthogonal_operator_game}
	Consider symmetric games with $n=2$ and $k_1=k_2=3$, i.e.,
	$G=[a, b, c, d, e, f, g, h, i, a, d, g, b,\\ e, h, c, f, i],$
	where $a,b,c,d,e,f,g,h,i\in\R$ are arbitrary.
	By using Theorem \ref{thm17:game_decomposition}, the orthogonal projection onto
	the potential subspace is calculated in Eqn. \eqref{eqn_projection_exam2}.
	\begin{figure*}
	\begin{equation}\label{eqn_projection_exam2}
		P_{\Gr_P}=\frac{1}{18}\begin{bmatrix}
	 14,  2,  2,  2, -1, -1,  2, -1, -1,  4, -2, -2, -2,  1,  1, -2,  1,  1\\
	   2, 14,  2, -1,  2, -1, -1,  2, -1, -2,  4, -2,  1, -2,  1,  1, -2,  1\\
	   2,  2, 14, -1, -1,  2, -1, -1,  2, -2, -2,  4,  1,  1, -2,  1,  1, -2\\
	   2, -1, -1, 14,  2,  2,  2, -1, -1, -2,  1,  1,  4, -2, -2, -2,  1,  1\\
	  -1,  2, -1,  2, 14,  2, -1,  2, -1,  1, -2,  1, -2,  4, -2,  1, -2,  1\\
	  -1, -1,  2,  2,  2, 14, -1, -1,  2,  1,  1, -2, -2, -2,  4,  1,  1, -2\\
	   2, -1, -1,  2, -1, -1, 14,  2,  2, -2,  1,  1, -2,  1,  1,  4, -2, -2\\
	  -1,  2, -1, -1,  2, -1,  2, 14,  2,  1, -2,  1,  1, -2,  1, -2,  4, -2\\
	  -1, -1,  2, -1, -1,  2,  2,  2, 14,  1,  1, -2,  1,  1, -2, -2, -2,  4\\
	   4, -2, -2, -2,  1,  1, -2,  1,  1, 14,  2,  2,  2, -1, -1,  2, -1, -1\\
	  -2,  4, -2,  1, -2,  1,  1, -2,  1,  2, 14,  2, -1,  2, -1, -1,  2, -1\\
	  -2, -2,  4,  1,  1, -2,  1,  1, -2,  2,  2, 14, -1, -1,  2, -1, -1,  2\\
	  -2,  1,  1,  4, -2, -2, -2,  1,  1,  2, -1, -1, 14,  2,  2,  2, -1, -1\\
	   1, -2,  1, -2,  4, -2,  1, -2,  1, -1,  2, -1,  2, 14,  2, -1,  2, -1\\
	   1,  1, -2, -2, -2,  4,  1,  1, -2, -1, -1,  2,  2,  2, 14, -1, -1,  2\\
	  -2,  1,  1, -2,  1,  1,  4, -2, -2,  2, -1, -1,  2, -1, -1, 14,  2,  2\\
	   1, -2,  1,  1, -2,  1, -2,  4, -2, -1,  2, -1, -1,  2, -1,  2, 14,  2\\
	   1,  1, -2,  1,  1, -2, -2, -2,  4, -1, -1,  2, -1, -1,  2,  2,  2, 14
	\end{bmatrix}
\end{equation}\end{figure*}
	Such games are potential iff, $P_{\Gr_P}G^T=G^T$ iff, $ c - b + d - f - g + h=0$.
	
	Besides, by Theorem \ref{thm17:game_decomposition}, such games are pure harmonic iff,
	$(I_{18}-P_{\Gr_P})G^T=G^T$ iff, $a=e=i=0$ and $b=g=-c=-d=-h=f$.
\end{example}

\begin{example}\label{exam3:orthogonal_operator_game}
	From the orthogonal projection $P_{\Gr_P}$ in Example \ref{exam2:game_subspace}, it follows that
	the Rock-Paper-Scissors game $G=[0,-1,1,1,0,-1,-1,1,0,0,1,-1,-1,0,1,1,\\-1,0]$
	satisfies that $P_{\Gr_P}G^T=0$, hence this game
	is pure harmonic. \cite[Example 27]{Cheng2014PotentialGame} shows  that this game is not
	potential. Note that for this game, $n=2$, $k_1=k_2=3$, so the pure harmonic game and the pure harmonic
	game studied in \cite{Candogan2011FlowsDecompositionofGames} are the same.
	\cite[Example 4.1]{Candogan2011FlowsDecompositionofGames} also shows that this game is pure harmonic.
\end{example}

\section{Conclusion}\label{sec:conclusion}

In this paper, we have given complete characterization for how to project a finite game onto the
subspaces of pure potential games, nonstrategic games, pure harmonic games, potential games
and harmonic games, respectively. Based on this, one can easily verify whether a finite game belongs to one of
these subspaces.

From Theorem \ref{thm4:game_decomposition}, no matter a finite game is potential, one can 
obtain a function from the set of strategy profiles to the reals. Then for a general finite game,
what is the meaning of such a function? This interesting problem is left for further study.

Not like potential games, harmonic games are young games. The results of this paper may provide
an effective way to the further study on
theory and potential applications of harmonic games.


\section*{Acknowledgment}

The author would like to express his sincere gratitude to Prof. Daizhan Cheng, Dr. Hongsheng Qi,
and Mr. Ting Liu from Key Laboratory of Systems and Control,
Academy of Mathematics and Systems Sciences,
Chinese Academy of Sciences
for fruitful discussions, especially for the simpler proof of Theorem
\ref{thm6:game_decomposition} and the appearance of Theorem \ref{thm19:game_decomposition}.

\end{document}